\documentclass[reqno,12pts]{amsart}
\usepackage{amssymb,amsmath,amsthm,amstext,amsfonts}
\usepackage{amsmath,amstext,amsthm,amsfonts}
\usepackage[dvips]{graphicx}
\usepackage[dvips]{graphicx}
\usepackage{color}

\makeatletter \@addtoreset{equation}{section} \makeatother

\renewcommand\thetable{\thesection.\@arabic\c@table}

\theoremstyle{plain}
\newtheorem{maintheorem}{Theorem}

\newtheorem{theorem}{Theorem }[section]
\newtheorem{proposition}[theorem]{Proposition}
\newtheorem{lemma}[theorem]{Lemma}
\newtheorem{corollary}[theorem]{Corollary}

\theoremstyle{definition} \theoremstyle{remark}
\newtheorem{remark}[theorem]{Remark}
\newtheorem{example}[theorem]{Example}
\newtheorem{definition}[theorem]{Definition}

\newcommand{\field}[1]{\mathbb{#1}}
\newcommand{\real}{\field{R}}

\renewcommand{\natural}{\field{N}}

\newcommand{\al} {\alpha}       
        
\newcommand{\ga} {\gamma}

\newcommand{\la} {\lambda}      \newcommand{\La}{\Lambda}

\newcommand{\diam}{\operatorname{diam}}

\newcommand{\ov}{\overline}

\renewcommand{\field}[1]{\mathbb{#1}}

\newcommand{\re}{\field{R}}

\renewcommand{\natural}{\field{N}}

\newcommand{\cC}{\mathcal{C}}

\newcommand{\cR}{\mathcal{R}}

\newcommand{\cF}{\mathcal{F}}

\newcommand{\cU}{\mathcal{U}}
\newcommand{\SU}{\mathcal{U}}

\newcommand{\ds}{\displaystyle}

\newcommand{\f}{\ds\frac}

\newcommand{\N}{\mathbb{N}}
\newcommand{\Z}{\mathbb{Z}}

\newcommand{\E}{\mathbb{E}}
\newcommand{\R}{\mathbb{R}}

\renewcommand{\L}{\mathcal{L}}
\newcommand{\Lt}{\tilde{\L}}

\newcommand{\A}{\mathcal{A}}

\newcommand{\D}{\mathcal{D}}
\newcommand{\F}{\mathcal{F}}

\renewcommand{\P}{\mathcal{P}}
\newcommand{\G}{\mathcal{G}}

\newcommand{\gm}{\gamma}
\newcommand{\Gm}{\Gamma}
\newcommand{\fhi}{\varphi}

\newcommand{\limn}{\ds\lim_{n \to \infty}}

\newcommand{\seta}{\rightarrow}

\renewcommand{\l}{\left}
\renewcommand{\r}{\right}

\newcommand{\norma}[1]{\| #1 \|}
\newcommand{\modulo}[1]{\l| #1 \r|}
\newcommand{\holder}[1]{\l| #1 \r|_\alpha}
\newcommand{\fecho}[1]{\overline{#1}}

\theoremstyle{plain}

\newtheorem*{teoremaB}{Theorem B}
\newtheorem*{teoremaC}{Theorem C}
\newtheorem*{theorem*}{Theorem}

\renewcommand{\gm}{\gamma}
\newcommand{\gmj}{\gm_j}
\newcommand{\gmt}{\widetilde{\gm}}
\newcommand{\gmtj}{\widetilde{\gm}_j}
\newcommand{\gmc}{\hat{\gm}}
\newcommand{\gmcj}{\hat{\gm}_j}
\newcommand{\mg}{\mu_{\gm}}
\newcommand{\mgj}{\mu_{\gm_j}}

\newcommand{\rj}{\rho_j}
\newcommand{\rl}{\rho^{,}}
\newcommand{\rlj}{\rho^{,}_j}
\newcommand{\rll}{\rho^{,,}}
\newcommand{\rllj}{\rho^{,,}_j}

\newcommand{\rbj}{\overline{\rho}_j}
\newcommand{\rbbj}{\overline{\overline{\rho}}_j}
\newcommand{\rt}{\widetilde{\rho}}

\newcommand{\rtj}{(\widetilde{\rho})_j}
\newcommand{\rc}{\hat{\rho}}
\newcommand{\rcj}{(\hat{\rho})_j}

\newcommand{\dn}{d\nu}
\newcommand{\dm}{d\mu}
\newcommand{\dmg}{d\mu_{\gm}}
\newcommand{\dmgj}{d\mu_{\gm_{j}}}
\newcommand{\dmgt}{d\mu_{\gmt}}
\newcommand{\dmgtj}{d\mu_{\gmt_{j}}}
\newcommand{\dmgc}{d\mu_{\gmc}}
\newcommand{\dmgcj}{d\mu_{\gmcj}}

\newcommand{\infimo}[2]{\ds\inf_{#1}\l\{ #2 \r\}}
\newcommand{\intg}[1]{\ds\int_{\gm} #1 \dmg}
\newcommand{\intgj}[1]{\ds\int_{\gmj} #1 \dmgj}
\newcommand{\intgt}[1]{\ds\int_{\gmt} #1 \dmgt}

\newcommand{\intgc}[1]{\ds\int_{\gmc} #1 \dmgc}
\newcommand{\intgcj}[1]{\ds\int_{\gmcj} #1 \dmgcj}
\newcommand{\Dnorm}[1]{\mathcal{D}_1(#1)}
\newcommand{\sumjap}{\ds\sum_{j=1}^p}
\newcommand{\sumjapg}{\ds\sum_{j=1}^{p_{\gm}}}

\begin{document}

\title{Statistical properties  of the maximal entropy measure for  partially hyperbolic attractors}

\author{Armando Castro and  Te\'ofilo Nascimento}

\address{Armando Castro, Departamento de Matem\'atica, Universidade Federal da Bahia\\
Av. Ademar de Barros s/n, 40170-110, Salvador-Ba, Brazil.}
\email{armando@impa.br}

\address{Te\'ofilo Nascimento, Departamento de Ci\^encias Exatas e da Terra - Campus II, Universidade do Estado da Bahia\\
Br 110, km 03, 48.040 -210, Alagoinhas-Ba, Brazil.}
\email{atnascimento@uneb.br}

\date{\today}

\begin{abstract}
We show the existence and uniqueness of the maximal entropy probability measure 
for partially hyperbolic diffeomorphisms which are semi-conjugate to 
nonuniformly expanding maps.
Using the theory of projective metric on cones we then
prove exponential decay of correlations for H\"older continuous observables and
the central limit theorem for the maximal entropy probability measure. 
Moreover, for systems derived from solenoid we also prove
the statistical stability for the maximal entropy probability measure.  
Finally, we use such techniques to obtain similar results
in a context containing partially hyperbolic systems derived from Anosov.
\end{abstract}


\maketitle
%
%

%

%
\section{Introduction}
The thermodynamical formalism from the statistical mechanics was introduced in Dynamical Systems by the former works  of Sinai, Ruelle and Bowen for uniformly hyperbolic maps and  H\"older potentials, in the beginning of the 70's.  Beyond the  uniformly hyperbolic context,  the theory is still quite incomplete. Several contribution do exist, for example \cite{BK98, BF09, Yur03, OV08, SV09, BF09, Sar99, Cas02, VV10, LM13, CV13, MT14}.

In the recent years, the thermodynamical formalism of a class of partial hyperbolic diffeomorphisms 
introduced by Alves, Bonatti, Viana \cite{ABV00} and Castro~\cite{Cas98} has been developed under some conditions
that resemble or may lead to some mostly expanding or mostly contracting assumption in the central direction. 

In the non-invertible setting this has been studied by Castro, Oliveira, Varandas and Viana 
\cite{OV08,VV10, CV13}.  Given 
a compact metric space $M$ and a local homeomorphism  $f:M \to M$ in  with Lipschitz inverse branches
that admit some expanding and some possibly contracting domains of invertibility it was proved in
 \cite{VV10} that for every H\"older continuous potential  $\phi$ satisfying a small variation condition there are finitely many ergodic equilibrium states for $f$ with respect to $\phi$. Furthermore,  the equilibrium states are absolutely continuous with respect to some conformal measure and there exists a unique equilibrium state provided that the dynamical system is topologically exact.
Later on, using a functional analytic approach by means of projective metrics techniques to the study of the
spectral properties of Ruelle-Perron-Frobenius operators on the space of $C^{r+\alpha}$ observables 
$(r\in \mathbb N, \alpha>0)$, Castro and Varandas ~\cite{CV13} presented a more general proof for the 
uniqueness of equilibrium states for this class of maps and deduced many statistical properties  as 
exponential decay of correlations, Central Limit Theorem, and also both statistical and spectral stabilities.

In this paper our motivation is to contribute to the study of the thermodynamical formalism of 
a large class of partially hyperbolic  diffeomorphisms with strong stable foliation. We  are interested in  two
different settings. The first setting consists in  partially hyperbolic 
diffeomorphisms which are semiconjugate to the class of local diffeomorphisms in \cite{CV13}. 
The second setting consists in a class of partially hyperbolic attractors exhibiting a  Markov
partition (whose iterates need not have diameters going to zero).
These settings
include many examples of partially hyperbolic diffeomorphisms that arise as local bifurcations of 
Axiom A diffeomorphisms and will be mostly expanding with respect to some conformal measure,
including a robust (open) class of systems derived from Anosov, introduced by Ma\~n\'e in \cite{Mane78}.
 
SRB measures for large classes of partially hyperbolic diffeomorphisms have been constructed 
by \cite{Car93, ABV00, BV00, Cas98} and existence and uniqueness of 
maximal entropy 
measures have been proved by Buzzi, Fisher, Sambarino, Vasquez \cite{BFSV12} for derived from Anosov 
diffeomorphisms, by Buzzi, Fisher \cite{BF13} 
for wide class of deformations of Anosov diffeomorphisms that include some examples by Bonatti and Viana of
robustly transitive non-partiallly hyperbolic diffeomorphisms, and
 by Ures \cite{Ur12} for partially hyperbolic diffeomorphisms of $\mathbb T^3$ homotopic to a 
hyperbolic automorphism. In most of these cases the approach is to establish a semiconjugacy between the
dynamical system and some uniformly hyperbolic one and prove that the points that remain in a non-hyperbolic
region do not contribute much for the topological entropy. More recently, Climenhaga, Fisher and 
Thompson \cite{CFT15} proved the uniqueness for equilibrium states  for some robust classes of
examples of \cite{Mane78} and \cite{BV00}. The drawback is that these methods are not enough to
deduce some good statistical properties for the original dynamical system, specially the exponential decay
of correlations. 
To illustrate this fact let us mention that in the case of nonuniformly expanding maps the Ruelle-Perron-Frobenius
transfer operator acts in the space of H\"older continuous functions and the dominant eigenvector of its adjoint operator leads to the measure of maximal entropy, while in the invertible  context any invariant measure is an
eigenvector for the adjoint operator. For that reason the method of invariant cones used in \cite{CV13} could not be applied here.
In fact, the results and their proofs in this paper here are independent from those in the above mentioned paper, 
except that we use the existence and uniqueness of the entropy maximizing measure there
to guarantee the uniqueness in this new context.  

So, to deduce exponential decay of correlations for the original dynamical systems we introduce a suitable
Banach space and prove that the transfer operator does preserve some cone of functions. The construction
of such cone of functions is done by constructing a family of probability measures on stable leaves that
is equidistributed and holonomy invariant. A very laborious  work is done in order to prove the invariance of such suitable 
cone of functions by the transfer operator and that the image of this by the transfer  operator has finite diameter 
in the  projective metrics, which implies that transfer operator is a contraction with respect to the projective metrics. 
From that and the duality properties of transfer and Koopman operators we derive the 
exponential decay of correlations and the Central Limit Theorem as a consequence.

It is worth to mention other recent works (e.g. \cite{BL12, Mel14, LT15})   concerning fast mixing of SRB and Gibbs measures, in some nonuniformly hyperbolic settings.  The techniques  used in such papers are either compactness arguments provided by Lasota-Yorke estimates, or
Young Towers \cite{You98} and operator renewal theory. Even 
though such techniques have the advantage to reach a kind of spectral gap for transfer operator rather directly, they 
need asymptotic assumptions, and stronger  transitivity assumptions than the Cone approach. 
For instance,  the nontransitive situation that we obtain by bifurcating the Manneville-Poumeau map 
so that we create a sink can not be properly worked out by a Lasota-Yorke approach. 
However, the method of invariant cones for transfer operators, used e.g. in \cite{CV13},
easily contemplates such example without 
any addititional hipothesis. 
The approach here also gathers the same advantage of a kind of mild transitivity assumptions such as in \cite{CV13}. 
So, our paper deals with different and robust classes of examples that are not under the 
hypotheses of the previous cited works.

This paper is organized as follows. 
In the next section, we give precise definitions of the family of partially hyperbolic diffeomorphisms that we consider
and state the main results. Some robust class of examples is also discussed.
In  sections 3 and 4,  we establish the existence and  uniqueness  of equlibrium states.
and, restricting to the  skew-products and derived from solenoid case, in section 5,
we also prove statistical stability of the equilibrium states, meaning that
the measure varies continuously in the weak$^*$ topology with the dynamics and the potential. 
In the remaining sections up to section 10,  we prove that  the  maximal entropy measure satisfies good statistical properties,
namely exponential decay of correlations and the Central Limit Theorem in the space of H\"older continuous observables. 
In the last section \ref{secanosov}, we apply the methods that we developed for the case of partially hyperbolic attractors with Markov partition,
including some robust classes of attractors derived from Anosov introduced by Ma\~n\'e \cite{Mane78}.

\section{Context and statement of the main results}
\label{contexto}

In this paper, we will work with two contexts of partially hyperbolic diffeomorphisms with strong
stable direction.
We deal with  
 partially hyperbolic systems that are semiconjugate to
  nonuniformly expanding endormorphisms (see \cite{CV13}) and 
with diffeomorphisms that include systems derived from Anosov.
Although both classes of dynamical systems presents a partially hyperbolic behaviour,
the study of their thermodynamical properties require different approaches due 
to crucial geometrical differences.

\vspace{0.3cm}
{\bf  First Setting.} 
Let $N$ be a connected compact Riemannian manifold, and let $g:N \to N$ be a \emph{local homeomorphism} with Lipschitz 
inverse branches. 
For that, we mean there  exists $L(x) \geq 0$ such that,  for all  $x\in N$ has a neighborhood 
$U_x \ni x$ such that $g_x:= g|_{U_x} : U_x \to g(U_x)$ is invertible
and 
\begin{equation}\label{CV}
d(g_x^{-1}(y),g_x^{-1}(z))
    \leq L(x) \;d(y,z), \quad \forall y,z\in g(U_x).
\end{equation}
Let us denote by $\deg(g)$ the degree of  $g$, which coincides with the number of preimages of any $x \in N$ by $g$.
We also assume that 
there exist $0<\lambda_u<1$ and an open region $\Omega \subset N$
such that \vspace{.1cm}
\begin{itemize}
\item[(H1)] $L(x)\leq L$ for  $x \in \Omega$ and
$L(x)< \lambda_u $ for $x\notin \Omega$, for some $L$ close to $1$.
\item[(H2)] There exists a covering $\cU$ of $N$ by injective domain of $g$, 
such that $\Omega$ can be covered by $q<\deg(g)$ elements of $\cU$. \vspace{.1cm}
\end{itemize}

Let $M$ be a compact invariant manifold,  and $f:M \to M$ a diffeomorphism onto its image.  
Suppose there exists a continuous and sujective  $\Pi: M \to N$  such that 
\begin{equation} \label{contextoprincipal}
\Pi \circ f = g \circ \Pi.
\end{equation}
Given $y \in N$, set $M_y= \Pi^{-1}(y)$.  Therefore, $M=\ds\bigcup_{y\in N}M_y$. Note that  $f(M_y) \subset M_{g(y)}$,
and also suppose that there exists $0 <\lambda_s <1$ such that
\begin{equation}
d(f(z),f(w)) \leq \lambda_s d(z,w)
\end{equation}
for all $z,w \in M_y$. 

As the maximizing entropy measure is  $f$-invariant, by  Poincar\'e's Recurrence Theorem such measure is
 supported in the attractor
$$
\Lambda:=\bigcap_{n=0}^{\infty}f^n(M).
$$
Note that $\Lambda$ is compact and invariant by $f$. So, it is sufficient to study the  dynamics of $f$ restricted to $\Lambda$. 

Given $x, y\in M$, write $\hat{x}:= \Pi(x)$, $\hat{y}:= \Pi(y)$.
We assume that there exist holonomies $\pi_{\hat{x},\hat{y}}: M_{\hat{x}}\cap \Lambda \to M_{\hat{y}}\cap \Lambda$ satisfying 
\begin{equation}
\f{1}{C}\l[d_N(\hat{x},\hat{y})+d_M(\pi_{\hat{x},\hat{y}}(x),y)\r] \leq d_M(x,y) \leq C\l[d_N(\hat{x},\hat{y})+d_M(\pi_{\hat{x},\hat{y}}(x),y)\r]
\end{equation}
for some constant $C>0$, and $d_M$, $d_N$ to be the metrics of $M$,$N$, respectively. For simplicity we shall write $d$ for any of such metrics.

We suppose such  holonomies are invariant  by $f$, that is, 
\begin{equation}
f\l(\pi_{\hat{x},\hat{y}}(z)\r)=\pi_{g(\hat{x}),g(\hat{y})}\l(f(z)\r) \label{fimcontexto}
\end{equation}
for all $z \in M_{\hat{x}}\cap \Lambda$.

\vspace{0.3cm}
{\bf Second Setting.}
Let $M$ be a compact Riemannian manifold and  $f:M \seta M$ be a $C^{1+}$diffeomorphism. 
Assume  that there exists 
a compact subset $\Lambda$ of  $M$ with the following properties:
\begin{enumerate}

	\item  There exists an open $f-$invariant neighborhood $Q$ of $\Lambda$, such that $f(\fecho{Q}) \subset Q$ and 
		$$
		\Lambda=\bigcap_{n=0}^{\infty}f^n(Q).
		$$
	\item  $\Lambda$ is partially hyperbolic, in the sense that there  exists a $Df$-invariant dominated splitting
		$$
			T_{\Lambda}M=E^{ss}\oplus E^{uc}, dim(E^{ss})>0
		$$
		of the tangent bundle restricted to $\Lambda$, such that, once fixed a Riemannian metrics in $M$ we have:

	\begin{enumerate}
		
		\item $E^{ss}$ contracts uniformly: 	
				$\norma{Df^n|E^{ss}_x} \leq C\lambda_s^n$
		
		\item $E^{uc}$ is dominated by $E^{ss}$:
				$\norma{Df^n|E^{ss}_x}\norma{Df^{-n}|E^{uc}_{f^n(x)}} \leq C \lambda_s^n$ 
				
	\end{enumerate}
	
	for all  $n \geq 1$ and $x \in \Lambda$,	with $0<\lambda_s<1$.
	
	\item There exists an $f$-invariant  center-unstable foliation  $\F^{uc}_{loc}$ of a neighborhood $\Lambda$,  which is tangent to
	the  center unstable subbundle $E^{uc}$ in $\Lambda$.  There is also an $f-$invariant  
	stable foliation $\F^s_{loc}$ tangent 
	to the stable  subbundle $E^{ss}$ in $\Lambda$.
	
	In order to proceed with our considerations on  the dynamics $f$, we recall the concept 
	of Markov Partition in this partially hyperbolic context. 

\begin{definition}
	We say that $R \subset \Lambda$ is a  Markov \textbf{proper rectangle}, if forall  $x$ and $y$ in $R$ there exists a unique
	 point $z:=[x,y] \in  R$ which is the intersection between the local (strong) stable manifold passing by  $x$, and the
	 local center unstable manifold passing by $y$. Moreover,
	$R$ is the closure of its interior (in the  relative topology of $\Lambda$) and, in particular, is closed.
\end{definition}

We observe that  the  boundary of Markov proper rectangles are  union  of local (strong) stable  manifolds
and local center-unstable manifolds. 

\begin{definition}
	A collection $\cR=\l\{R_1, \cdots,R_p\r\}$ of proper rectangles is a  {\em Markov Partition} for $f$ restricted to $\Lambda$, if:
	\begin{enumerate}
		\item $\ds\Lambda=\bigcup_{i=1}^p R_i$;
		\item $int\l(R_i\r)\cap int\l(R_j\r)=\emptyset$ for $i \neq j$;
		\item If $\gm$ is the  intersection of  a local (strong) stable  manifold with  $R_i$ and 
		 $f(\gm) \cap R_j \neq \emptyset$ then $f(\gm) \subset R_j$.
		Analogously, if $\Gm$ is the  intersection of a  local center unstable manifold with $R_i$ and 
		$f^{-1}(\Gm) \cap R_k \neq \emptyset$ then $f^{-1}(\Gm) 		\subset R_k$.
	\end{enumerate}
	
\end{definition}
	
	\item  $f$ restricted to $\Lambda$ admits a Markov partition $\cR=\l\{R_1,\cdots,R_p\r\}$, $p \geq 2$ with the (mild) mixing property: 
	given $i,j \in \l\{1, \cdots, p\r\}$, there exists $n_0 \geq 1$ such that
	
		$$ 
		    f^n(R_i) \cap R_j \neq \emptyset, \forall n \geq n_0.
		$$
	
	We distinguish two kinds of   rectangles in $\cR$ according  to its behavior in the  direction $E^{uc}$. Fixed $0<\zeta < 1$,  
	we say $R_i \in \cR$	is a {\em good rectangle} if 
		$$
			\norma{Df|_{E^{uc}_x}}^{-1}\leq \zeta
		$$
	for all $x \in R_i$. That is, $E^{uc}$ expands uniformly in $R_i$, for one iterate. The other rectangles will be 
	called {bad rectangles}.
	
		\item There  exists at least one good rectangle and for all  $x$ in a bad  rectangle
			$$
				\norma{Df|_{E^{uc}_x}}^{-1}\leq L
			$$
	for some $L \geq 1$ close to $1$ (depending on $\zeta$ and the combinatorics of the partition). 
\end{enumerate}

\subsection{Statement of the main results}

We recall  the definition of topological entropy due to  Bowen, using $(n,\epsilon)$-separable sets. A  compact set $K$ contained in a metric space $(X,d)$ is $(n,\epsilon)$-separable  if 
$$
\forall x,y \in K, x\neq y, \max\l\{d(f^j(x),d(f^j(y)); j=0,\cdots,n-1\r\}>\epsilon
$$
We denote by $S(n,\epsilon,K)$ the greatest cardinality of  a $(n,\epsilon)$-separate subset  of $K$. The {\em relative entropy } of $f$ with respect to a (not necessarily invariant) compact $K \subset X$, is given by 
$$
h(f,K):=\ds\lim_{\epsilon \seta 0}\limsup_{n \seta \infty}\f{1}{n}\log S(n,\epsilon,K). 
$$
For a uniformly continuous map $f: X \seta X$, ($X$ not  necessarily compact), the {\em topological entropy} is defined by
$$
h(f):=\sup\l\{h(f,K); K \text{ compact }\r\}
$$
In our  context $X= \Lambda$ is a compact set, and $f$ is automaticaly uniformly continuous.
We also have by \cite{ W93} that 
$h(f)=h(f,X)$ does not depend on the metrics.

For an invariant measure $\mu$, we also recall the definition by Kolmogorov of its metric entropy $h_\mu(f)$.
Given a probability space $(X,\mathcal{B},\mu)$, if $\mu$ is $f$-invariant, 
we  define the entropy of a finite of a finite partition $\P$ of $X$ by:
$$
h_{\mu}(\P):= -\ds\sum_{P \in \P} \mu(P)\log \mu(P).
$$

Then the entropy of a partition with respect to $f$ is 
$$
h_{\mu}(f,\P):= \lim_{n \seta \infty}\f{1}{n}h_{\mu}(\P \vee f^{-1}(\P) \vee \cdots \vee f^{n-1}(\P)).
$$
and the metric entropy of $f$ with respect to $\mu$ is given by
$$
h_{\mu}(f):=\sup_{\P} \l\{h_{\mu}(f,\P)\r\}.
$$
Denote by  $ \mathcal{M}^1_f(X)$ the set of all $f-$invariant probabilities.
The {\em variational principle} stablishes, that for a continuous map $f$ on a compact metric space $X$, 
the equation
$$
h(f)=\sup\l\{h_{\mu}(f); \mu \in \mathcal{M}^1_f(X) \r\}
$$
holds.
We say that an invariant probability $\mu$ is a {\em maximal entropy measure } for  $f$ if 
$
h(f)=h_{\mu}(f).
$
We now state the main results in this work:
\begin{maintheorem}\label{thm.max.entropy}{\textbf{(Existence and Uniqueness of Maximal Entropy measure.)}}
Let $f:\Lambda \to \Lambda$ a diffeomorphism  in the first setting, as described in section  \ref{contexto} (that is, the conditions
given by equations \ref{contextoprincipal} through \ref{fimcontexto}).
Then, there exists a unique  maximal entropy measure $\mu$ for $f$. 
\end{maintheorem}

As a by-product of the proof we also obtain
\begin{corollary}{\textbf{(Statistical Stability in the Derived from Solenoid case.)}}
Let $f_n$ be a sequence of derived from solenoid diffeomorphisms (see  example   \ref{exskew3} ) 
and call $\mu_n$ the maximal entropy probability measure for $f_n$.
If $f_n \to f$ in the $C^1$-topology, then $\mu_n$ converges to the maximal entropy probability measure for $f$ in the 
weak-* topology. 
\end{corollary}

Using the theory of projective metrics over invariant cones,  we prove:

\begin{maintheorem}{\textbf{(Exponential Decay of Correlations)}} \label{thB}
The maximal measure entropy $\mu$ for $f:\Lambda \to \Lambda$ has exponential decay of correlations for 
H\"older continuous observables, that is, 
there exists some $0<\tau<1$ 
such that for $\alpha$-H\"older continuous $\fhi,\psi$ there exists $K(\varphi,\psi)>0$ satisfying
\begin{equation*}
\left|\int (\varphi\circ f^n) \psi d\nu - \int \varphi d\nu\int \psi d\nu\right|
	\leq K(\varphi,\psi)\cdot\tau^n,
	\quad \text{for all $n\ge 1$}.
\end{equation*}

\end{maintheorem}

For the maximal entropy measure $\mu$ the following theorem also holds:

\begin{maintheorem}{\textbf{(Central Limit Theorem)}}

Let $\mu$ be the maximal entropy measure for $f:\Lambda \to \Lambda$, as in (\ref{contextoprincipal}) and let $\varphi$ be a H\"older continuous function. If
$$
\sigma_\varphi^2:=\int \phi^2 d\mu + 2\sum\limits_{j=1}^{\infty}\int \phi\cdot (\phi\circ f^j) \, d\mu,
\quad \text{ with } \quad \phi=\varphi-\int \varphi \, d\mu,
$$
then 
 $\sigma_\varphi<\infty$ and $\sigma_\varphi=0$ if, and only if, $\varphi=u\circ f - u$ for some $u \in L^2(\mu)$. 
Moreover, if $\sigma_\varphi>0$ then,  for all interval $A\subset\real$
\begin{equation*}
	\lim_{n\to\infty}\mu\left(x\in M: \frac{1}{\sqrt{n}}\sum\limits_{j=0}^{n-1}
	\left(\varphi(f^j(x))-\int \varphi d\mu\right)\in A\right)= \frac{1}{\sigma_\varphi\sqrt{2\pi}}
	\int_A e^{-\frac{t^2}{2\sigma_\varphi^2}} dt
\end{equation*}
holds.
\end{maintheorem}

In what follows, we shall describe the results the results for the class of partially hyperbolic 
diffeomorphisms considered in the Second Setting.
In this other context, we construct a dominant  eigenmeasure $\mu$ for the dual 
of the transfer operator acting in a suitable space of distributions.
We then prove:

\begin{maintheorem} The measure $\mu$ exhibits exponential decay of correlations in the space of H\"older continuous  observables \label{thD}. Furthermore,  the Central Limit Theorem holds for the measure $\mu$.
\end{maintheorem}

We note that, even in the second setting,  the Markov Partition permits us to construct a quotient
map from the original one (see page \pageref{defquot} for the precise definition).
In the cases in which the maximizing entropy measure exists and is unique for the quotient system,
one can repeat the arguments in Theorem \ref{thm.max.entropy} to conclude that the measure $\mu$ is
the (unique) maximal entropy measure for the system $(f, \Lambda)$.

In particular, by using \cite{LSV98}, we obtain:

\begin{corollary} \label{CorLSV}
Let $f$ be a system satisfying (1) through (5) of the Second Setting.
Suppose also that the center-unstable spaces of $f$ are one-dimensional.
Then $f$ has a unique maximizing entropy measure, which has
exponential decay of correlations and satisfies the Central Limit Theorem for 
H\"older continuous observables. 
\end{corollary}

\subsection{Some Examples}
\label{secexample}

\ 

Let us start with examples of the first setting.

\begin{example}
  \label{exskew1}
The most simple family of examples is a skew-product obtained from a map $g:N \seta N$ as in \cite{CV13} 
(this means that $g$ can be taken in a robust class of nonuniformly expanding maps that, in particular, includes all expanding maps)
and an endomorphism
 $\Phi:N\times K \seta K$, by the formula
$$
\begin{array}{rll}
f:& N\times K &\seta N \times K \\
&(x,y)&\mapsto(g(x),\Phi(x,y))
\end{array}
$$
such that $f$ is a diffeomorphism onto its image, and for each $x \in N$,  $\Phi(x,\cdot):K \seta K$ is a  $\lambda_s$-contraction. 
In such case, $\Pi$ is the canonical projection in the first coordinate, and $N\times K= \ds\bigcup_{x \in N}K_x$, where
 $K_x=\{x\} \times K$ forall $x \in N$.

\end{example}

\begin{example}
\label{exskew2}
As a subexample, we may take the solenoid generated in the  solid torus  $S^1 \times D$.  We define $f$ by
$$
\begin{array}{rll}
f:& S^1\times D &\seta S^1\times D \\
&(\theta,z)&\mapsto(g(\theta),\fhi(\theta)+A(z))
\end{array}
$$
where $g$ is the  Manneville-Pomeau map given by

\begin{equation*}\label{eq. Manneville-Pomeau}
g(\theta)= \left\{
\begin{array}{cl}
\theta (1+2^{\alpha} \theta^{\alpha}) & ,\mbox{ if}\; 0 \leq \theta \leq \frac{1}{2}  \\
(\theta-1)(1+2^\alpha(1-\theta)^\alpha)+1 & ,\mbox{ if}\; \frac{1}{2} < \theta \leq 1
\end{array}
\right.
\end{equation*}
where $\al \in (0,1)$, $\fhi$ is a local diffeomorphism and $A$ is a contraction.
\end{example} 

\begin{example} 
 One can modify the examples above in order to obtain robust (containing an open set) classes
of examples. These are examples derived from solenoid-like systems. \label{exskew3} 
For sake of simplicity, we will give a construction in dimension four, which can be easily adapted
to similar higher dimensional examples.

Let us begin with a solenoid-like 
$C^2-$skew-product hyperbolic diffeomorphism
$f_0:T^2 \times D    \to T^2\times D$ similar to the examples 1 and 2 above.
We suppose that 
$$
\begin{array}{rll}
f_0:& T^2\times D &\seta T^2 \times D \\
&(x,y)&\mapsto(g_0(x),\Phi_0(x,y))
\end{array}
$$
is such that $g_0$ is an expanding map.

 We suppose
that the norm of $Df_0$ along
the stable subbundle and the
norm of $Df_{0}^{-1}$ along the
unstable bundle are bounded 
by a constant $\la_0< 1/3$.
Let $p$ be a fixed point of
$f_0$ and
let $\delta > 0$ be
a small constant.
Denote $V_0= B(p, \delta/2)$. Then,
in the  same manner as in \cite{Cas02}, 
we deform $f_0^{-1}$ inside $V_0$
by a isotopy obtaining a continuous family
of maps  $f_t, 0< t< 2$ in
such a way that
\begin{itemize}
\item[i)] The continuation $p_{f_{t}}$ 
of the fixed point $p$ goes 
through some generic bifurcation such as 
a flip bifurcation or a Hopf bifurcation. 
Points of different indexes appear in a transitive attractor for values of $t$ between $1$ and $2$
(staying all the time inside $V_0$). For $t= 1$
we have the first moment of the Hopf (or flip) bifurcation,
with $f_1$ conjugated to $f_0$. We  
suppose that  the derivative
$Df_1|_{E^{cu}}$ 
does not contract vectors.
In the case of Hopf bifurcation, we suppose that 
$Df_t|_{E^{cu}}(p_{f_t})$ 
exhibits complex eigenvalues, for all $t$;
\item[ii)] In the process, there always exist a strong-
stable cone field $C^{ss}$ 
(cf. \cite{Vi97} for definitions)
and a center-unstable cone field $C^{cu}$, defined everywhere,
such that  $C^{cu}$ contains the unstable direction
of the initial map $f_0$;
We also suppose that there exists a continuation of the torus $T^2 \times \{0\}$ which is $f_0$-invariant and normally hyperbolic.
So, for each $t \in [1, 2]$ there exists a  $f_t$-invariant manifold $T_t$ that is the normally hyperbolic continuation 
of  $T^2 \times \{0\}$.
\item[iii)] Moreover, the width of the cone fields
$C^{ss}$ and $C^{cu}$ are bounded by a small constant
$\alpha > 0$.
\item[iv)] There exist a constant $\sigma> 1$ and
a neighbourhood $V_1 \subset V_0 \cap W^s(p)$,
such that $J^c= \|det Df_{t}^{-1}|_{E^{cu}}\| > \sigma$
outside $V_1$;
\item[v)] The maps $f_{t}^{-1}$ is $\delta-C^0$ close 
to $f_0^{-1}$ outside $V_0$ so that $\|(Df_1^{-1}|_{E^{cu}})\| < \la_0 < 1/3 $ outside $V_0$.
\end{itemize}

Note that the properties 
stated in conditions i) 
through v), which are valid
for $f_t, 0 \leq t \leq 2$, 
are also valid 
for a whole $C^1$-neighbourhood $\SU$ of the set
of diffeomorphisms 
$\{f_t, 0\leq  t \leq 2\}$. 
In particular, by \cite{HPS77} conditions i) through
iii) imply that any $f \in \SU$ has an invariant
central foliation, since the central cone field enables
us to define a graph transform associated to it, 
with domain in the space of foliations tangent
to $C^{cu}$, which is not empty, since the
unstable foliation of $f_0$ is tangent to it.
On the other hand, all $f \in \SU$ also
exhibits a strong stable foliation varying
continuously with the diffeomorphism.
 
As a consequence of lemma 6.1 of \cite{BV00}
there is a $C^1$-neighbourhood $\SU_1 \subset \SU$
of the set $\{f_t, 1< t \leq 2 \}$ such that
for all $f \in \SU_1$, $\La= T^n$ 
is a partially hyperbolic attractor, which is
not hyperbolic, because it is transitive and
contains points with different
indexes. 

One can embed $T^2 \times D$ as a subset of $T^4$. 
So, it is easy to extend $f_t$ above to $T^4$ in a manner that each 
$f_t$ is hyperbolic (and structurally stable) outside  $T^2 \times D$. So, we
will assume each $f_t$ defined in $T^4$ in such way.

Now take $f$ in some small ball $B= B(f_1, \delta'),
\delta' < \delta/ 2$.
Suppose also that $\delta'$ is sufficiently
small such  that
all diffeomorphism in $B(f_1, \delta') \subset \SU$
is partially hyperbolic.
So,  if $\delta'> 0$ is small,
$B(f_1, \delta')$ is an open set
of diffeomorphisms of  $T^4$ satisfying the conditions in section \ref{contexto}.

\end{example} 

\begin{corollary} \label{exemp}
There exists an open set of non-hyperbolic
diffeomorphisms $f: T^4 \to T^4$ satisfying conditions expressed by equations \ref{contextoprincipal} through  \ref{fimcontexto}. 
\end{corollary}
\begin{proof}
Just take the open set of diffeomorphisms
$\SU_2= \SU_1 \cap B(f_1, \delta')$,
$\delta'$ as in the proposition above.
Conditions
in equations \ref{contextoprincipal}-\ref{fimcontexto} fit for every
diffeomorphism in a ball $B(f_1, \delta')$.
\end{proof}

\begin{example}[Derived from Anosov]  \label{exsecond}
Ma\~n\'e \cite{Mane78} has introduced a robust class of partially hyperbolic attractors 
by a pitchfork or Hopf bifurcation  of some periodic orbit. In \cite{Cas02} the author
proved that a robust class of maps satisfying conditions  (1) through (5), can be obtained. 
In fact, let $g:M \to M$ as in example 2)  in \cite{Cas02} of a system derived from Anosov.  Taking $f=g^{-1}$, the strong unstable ($E^{uu}$) and center stable ($E^{cs}$) subbundles of  $g$ 
become respectively   the strong stable subbundle ($E^{ss}$) and center-unstable subbundle ($E^{uc}$) for $f$, and such $f$ 
satisfies (1) through (5).  

\end{example}

\section{Construction of  the Maximal Entropy Measure}
Due  to the contraction in the stable foliation,  the  dynamics of  distinct orbits  of $f:M \to M$ will be determined by  the dynamical behavior of the map $g:N \to N$. As seen in \cite{CV13}, such map $g$ has only a  unique  maximal entropy measure, which we will denote by $\nu$.  
 
 We start the construction of the maximal entropy measure for $f$ by definining it on measurable sets 
  of the form $\Pi^{-1}(A)$, where  $A$ is a Borelian set of $N$.
  
Since $\Pi$ is a semiconjugation, by \cite{W93} one obtain that,
$$
h(f) \geq h(g).
$$
Moreover, due to Bowen \cite{Bow71} it follows that
$$
h(f) \leq h(g)+\sup\{h(f,\Pi^{-1}(y)); y \in N \}
$$
We now prove that  $h(f,\Pi^{-1}(y))=0$ for all $y \in N$. Indeed,  since $f:M_y \to M_{g(y)}$ is a $\lambda_s$-contraction, given $\epsilon>0$, 
the only   $(n,\epsilon)$-separate subsets  restricted to $M_y$ are unitary subsets. As $\Pi^{-1}(y)$ can be writen as  a union of $m(\epsilon) \in \N$ balls of  $\epsilon$-diameter, we conclude that the  cardinality of any  $(n,\epsilon)$-separate subset  of $\Pi^{-1}(y)$ 
is at most $m(\epsilon)$. 
By the definition entropy due to Bowen, this implies $h(f,\Pi^{-1}(y))=0$ for all $y \in N$. 
Therefore,  $h(f) \leq h(g)$, and so $h(f)= h(g)$. 

This allows us to construct the maximal entropy measure for  $f$ from the one for $g$. In fact, denote by $\nu$ the unique maximal entropy measure built in \cite{CV13}. Due to the variational principle and the fact of $h(f)= h(g)$, it follows that $h_\nu(g)$ , is greater than, or equal to the metric entropy of any  $f-$invariant
probability.  So, for the proof of existence part of the statement, it is sufficient to  obtain an $f$-invariant probablity $\mu$, whose metric entropy with respect to  $f$ is greater or equal than $h_\nu(g)= h(g)$. 

For that purpose, let $\Pi_{\Lambda} = \Pi|_\Lambda$. Let $\A_{N}$ be the Borel $\sigma$-algebra on $N$. Clearly, $\A_0:=\Pi_{\Lambda}^{-1}(\A_{N})$ is a $\sigma$-algebra on $\Lambda$. 
Since $f$ is a bijection  in $\Lambda$ and $\Pi_{\Lambda} \circ f= g \circ \Pi_{\Lambda}$, we have 
$$
A=\Pi_{\Lambda}^{-1}(B)=f\circ \Pi_{\Lambda}^{-1} \circ g^{-1}(B).
$$
As $g^{-1} (B)$ belongs to $\A_N$, it follows that $\A_0 \subset f(\A_0)$ and therefore $\A_n:=f^n(\A_0)$ is a sequence of $\sigma$-algebras such that $\A_0 \subset \A_1 \subset \cdots \subset \A_n \subset \cdots$.  Define $\mu_n:\A_n \seta [0,1]$ by
$ 
\mu_n(f^n(A_0))=\nu(\Pi_\Lambda(A_0)),
$
for all $A_0 \in \A_0$. Note that  $\mu_n$ is an $f$-invariant probability for all $n \in \N$. In fact, given $A=f^n(A_0)$, where
 $A_0=\Pi_{\Lambda}^{-1}(B)$ and $B \in \A_N$, due to the $g$-invariance of $\nu$ and the surjection of maps $g$ and $\Pi_\Lambda$, we have:
$$
\begin{array}{rcl}
\mu_n(f^{-1}(A))&=&\mu_n(f^{-1}(f^n(A_0)))=\, \mu_n(f^n(f^{-1}(A_0)))\\
&=&\nu(\Pi_\Lambda(f^{-1}(A_0)))=\, \nu(\Pi_{\Lambda}(f^{-1}\circ \Pi_{\Lambda}^{-1}(B) ))\\
&=&\nu(\Pi_{\Lambda}(\Pi_{\Lambda}^{-1}\circ g^{-1}(B) ))=\, \nu(g^{-1}(B))\\
&=&\nu(B)=\, \nu(\Pi_{\Lambda}(A_0))= \, \mu_n(f^n(A_0))= \, \mu_n(A)\\
\end{array}
$$

Now, as $\A_n \subset \A_{n+1}$, $\A:= \ds\bigcup_{n=0}^{\infty} \A_n$ is an algebra in $\Lambda$.

Then we define $\mu: \A \seta [0,1]$ the probability such that  $\mu(A)=\mu_n(A)$ if $A \in \A_n$. By the standard
measure theory arguments(see \cite{ Mane}), $\mu$ is $\sigma$-aditive.
Moreover, $\mu$ is an $f$-invariant probability, as $\mu_n$ are $f$-invariant probabilities. 
It rests to  prove that the smallest  $\sigma-$algebra that contains $\A$ is the Borel $\sigma-$algebra. 

For that, it is  sufficient to see that $\A$ contains a sequence of partitions whose diameter goes to zero. 

This is because $f:M_y \to M_{g(y)}$ is a $\lambda_s$-contraction. 

In fact, for each $n \in \N$, by the continuity of  $g^n$,  there exists $\delta(n)>0$ such that $d(z,w)<\delta(n)$ implies $d(g^n(z),g^n(w))<\lambda_s^n$, for all $z,w \in N$. Taking $\P^0_n$ a  partition of $N$ whose diameter is less than $\delta(n)$, we define a sequence of partitions of $\Lambda$ by 
\begin{equation}\label{ParticaoDiametroToZero}
\P_n:=f^n\l(\Pi_{\Lambda}^{-1}\l(\P^0_n\r)\r)
\end{equation}
Clearly, $\diam(\P_n) \to 0$ as $n \to +\infty$. Indeed,  given $\bar{x},\bar{y}$ in the same element of $\P_n$, writing $\bar{x}=f^n(x)$ and $\bar{y}=f^n(y)$ we have $\hat{x}= \Pi(x),\hat{y}= \Pi(y) \in P^0_n$. Therefore, noting that 
$g^n(\hat{x})=g^n(\Pi(x))=\Pi(f^n(x))=\tilde{x}$ and $g^n(\hat{y})=g^n(\Pi(y))=\Pi(f^n(y))=\tilde{y}$ we obtain
$$
\begin{array}{rcl}
d(f^n(x),f^n(y)) &\leq & C \l[ d(\tilde{x},\tilde{y}) + d(\pi_{\tilde{x},\tilde{y}}\circ f^n(x),f^n(y))\r] \\
&= & C \l[ d(g^n \circ \Pi(x),g^n \circ \Pi(y)) + d(f^n (\pi_{\hat{x},\hat{y}}(x)),f^n(y))\r] \\
& \leq & C \l[\lambda_s^n + \lambda_s^n d(\pi_{\hat{x},\hat{y}}(x),y)\r] \leq \,  C \l[1+diam(M)\r]\lambda_s^n.\\

\end{array}
$$
By a slight abuse of notation, we also write $\mu$ for its natural extension to the Borel $\sigma$-algebra of $M$. 

Now we prove that $\mu$ is a maximizing entropy measure for $f$, by proving that $h_\mu(f) \geq h_\nu(g)$. 
Denote by $B^n_\epsilon(g,y_0)$ a $(n,\epsilon)$-dynamical ball  of $g$ around $y_0 \in N$, that is, the set of points $y \in N$, 
such that $d(g^j(y),g^j(y_0)) < \epsilon, \forall j \in \{0,\cdots,n-1\}$. 
Due Brin-Katok Theorem, $\nu$-a.e. point $y \in N$,
$$
h_{\nu}(g)=\lim_{\epsilon \to 0} \limsup_{n \seta \infty} \f{1}{n}\log \f{1}{\nu\l(B^n_\epsilon(g,y)\r)}
$$
holds.

Take now $B^n_\epsilon(f,x)$ the $(n,\epsilon)$ dynamical ball of $f$ restricted to $\Lambda$ at $x \in \Lambda$. 
By the uniform continuity of $\Pi$, given $\epsilon >0$ there exists $0<\delta<\epsilon$ such that  $\Pi(B_{\delta}(w)) \subset B_{\epsilon}(\Pi(w))$ 
for all $w \in M$. Note that $B^n_\delta(f,x) \subset \Pi_\Lambda^{-1}\l(B^n_\epsilon(g,y)\r)$ for all $x \in \Pi_\Lambda^{-1}(y)$. 

In fact, given $z \in B^n_\delta(f,x)$ we shall prove that $\Pi(z) \in B^n_\epsilon(g,y)$. As $\Pi(x)=y$ we have for all $j \in \l\{0, \cdots, n-1\r\}$
$$
\begin{array}{rcl}
d(g^j\circ\Pi(z),g^j(y))&=&d(g^j\circ\Pi(z),g^j\circ\Pi(x))= d(\Pi\circ f^j(z),\Pi\circ f^j(x))< \epsilon.\\
\end{array}
$$
Therefore
$$
\mu\l(B^n_{\delta}(f,x)\r) \leq \mu\l(\Pi_\Lambda^{-1}\l(B^n_\epsilon\l(g,y\r)\r)\r)=\nu\l(B^n_\epsilon\l(g,y\r)\r)
$$
and since  $\delta \to 0$ as $\epsilon \to 0$ we obtain
$$
h_{\nu}(g) \leq \lim_{\delta \to 0} \limsup_{n \to \infty} \f{1}{n}\log \f{1}{\mu\l(B^n_{\delta}(f,x)\r)}
$$
for $\mu-$a.e. $x \in \La$.
So,
$$
\begin{array}{rcl}
h_{\mu}(f)&=&\ds\int_{\Lambda} \lim_{\delta \to 0} \limsup_{n \to \infty} \f{1}{n}\log \f{1}{\mu\l(B^n_\delta(f,x)\r)} d\mu \geq \ds\int_{\Lambda} h_{\nu}(g) d\mu =  \, h_{\nu}(g)
\end{array}
$$
and we conclude that $h_{\mu}(f) \geq h_{\nu}(g)=h(g)=h(f)$, which is the equivalent to say 
that  $\mu$ is maximal entropy measure for $f$.

\section{Uniqueness of Maximal Entropy Measure}

\label{secuniqueness} 
Now we prove the  uniqueness of maximal entropy measure  for $f$ built  in the last section. For such purpose, we use the 
 uniqueness of the maximal entropy measure for $g$, provided by \cite{CV13}. Suppose that $\mu_1$ is another invariant maximal entropy
 measure for $f$, different to $\mu$. Let $\nu_1:=\l(\Pi_{\Lambda}\r)_*\mu_1$, the push-forward of  $\mu_1$.
 
  We claim  that since $\mu_1$ is different to  $\mu$, it follows that $\nu_1$ is different to $\nu$. 
  Indeed, since $\mu_1 \neq \mu$, $\mu_1(A) \neq \mu(A)$ for some $A \in \A=\A_0\cup f(\A_0)\cup \cdots \cup f^n(\A_0)\cup\cdots$. The fact that such algebras on $\P(\Lambda)$ are nested implies that exist $A_0 \in \A_0$ and $n \in \N$ such that $f^n(A_0)=A$. By the definition
  of  $\A_0$, there exists $B_0 \in \A_N$ such that  $\Pi^{-1}_{\Lambda}(B_0)=A_0$. We now observe that, on one hand,   
$$
\nu_1(B_0)=\l(\Pi_{\Lambda}\r)_*\mu_1(B_0)=\mu_1(\Pi^{-1}_{\Lambda}(B_0))=\mu_1(A_0)=\mu_1(f^n(A_0))=\mu_1(A)
$$
and on the other hand, 
$$
\nu(B_0)=\nu(\Pi_{\Lambda}(A_0))=\mu(A_0)=\mu(f^n(A_0))=\mu(A).
$$
So, $\nu_1 \neq \nu$. By the $f$-invariance of $\mu_1$ it follow that  $\nu_1$ is $g$-invariant.

Let us prove that $\nu_1$ is a maximal entropy measure for $g$, which is a contradiction, since by \cite{CV13}, such 
probability is unique. For that, it is sufficient to prove that  $h_{\nu_1}(g) \geq h_{\mu_1}(f)$, since $h_{\mu_1}(f)= h(f)= h(g)$.

In fact, we may suppose that the sequence $\P_n=f^n\l(\Pi_{\Lambda}^{-1}\l(\P^0_n\r)\r)$, in \ref{ParticaoDiametroToZero}, is such that $\P_0 \leq \P_1 \leq \cdots \leq \P_n \leq \cdots$ and  as $\ds\bigcup_{n=0}^{\infty}\P_n$ generates the Borel $\sigma$-algebra of $\Lambda$,
we obtain
$$
h_{\mu_1}(f)=\sup_n\l\{h_{\mu_1}(f,P_n)\r\}.
$$
Therefore, for all $\epsilon>0$ there exists $n \in N$ such that
$$
h_{\mu_1}(f,P_n) \geq h_{\mu_1}(f)-\epsilon.
$$
However, if follows from the definition of  $\nu_1$ that for all $n \in \N$
$$
h_{\nu_1}(g,P^0_n)=h_{\mu_1}(f,\Pi_{\Lambda}^{-1}\l(P^0_n\r)).
$$
Indeed,  for a partition $\P$ we have
$$h_{\nu_1}(g,\P)= \ds\lim_{m \to \infty} \f{1}{m} h_{\nu_1}\l(\P \vee g^{-1}(\P)\vee \cdots \vee g^{-(m-1)}(\P)\r)$$
Due to the definition of $\nu_1$ and the semiconjugation between  $f$ and $g$ we obtain
$$
\begin{array}{rcl}
\nu_1\l(\ds\bigvee_{j=0}^{m-1}g^{-j}(P_{i_j})\r)&=&\mu_1\l(\Pi^{-1}_{\Lambda}\l(\ds\bigvee_{j=0}^{m-1}g^{-j}(P_{i_j})\r)\r)\\
&&\\
&=&\mu_1\l(\ds\bigvee_{j=0}^{m-1} f^{-j}\l(\Pi^{-1}_{\Lambda}(P_{i_j})\r)\r)\\
\end{array}
$$
which guarantees $h_{\nu_1}\l(\ds\bigvee_{j=0}^{m-1}g^{-j}(\P)\r)=h_{\mu_1}\l(\ds\bigvee_{j=0}^{m-1} f^{-j}\l(\Pi^{-1}_{\Lambda}(\P)\r)\r)$ 
and so, we have
$$h_{\nu_1}(g,\P)=h_{\mu_1}(f,\Pi_{\Lambda}^{-1}\l(\P\r)).$$
From the  $f$-invariance of $\mu_1$ it follows that
$$
h_{\mu_1}(f,\Pi_{\Lambda}^{-1}\l(\P^0_n\r))=h_{\mu_1}(f,\P_n)
$$
because $P_{n_j} \in \P_n$ if and only if there exist $P^0_{n_j} \in \P^0_n$ such that $P_{n_j}=f^n(\Pi^{-1}_{\Lambda}(P^0_{n_j}))$. 
Therefore
$$
\begin{array}{rcl}
\mu_1\l(\bigvee_{j=0}^{m-1} f^{-j}(P_{n_j})\r)&=&\mu_1\l(\bigvee_{j=0}^{m-1} f^{-j}\l(f^n\l(\Pi^{-1}_{\Lambda}(P^0_{n_j})\r)\r)\r)\\
&=&\mu_1\l(\bigvee_{j=0}^{m-1} f^{-j}\l(\Pi^{-1}_{\Lambda}(P^0_{n_j})\r)\r).\\
\end{array}
$$
We then obtain that for all $\epsilon > 0$ there exists $n \in \N$ such that
$$
\begin{array}{rcl}
h_{\nu_1}\l(g\r)& \geq & h_{\nu_1}\l(g,P^0_n\r) = h_{\mu_1}\l(f,\Pi^{-1}_{\Lambda}P^0_n\r)\\
& = & h_{\mu_1}\l(f,P_n\r) \geq  h_{\mu_1}(f)-\epsilon\\

\end{array}
$$
and this proves that $h_{\nu_1}(g)\geq h_{\mu_1}(f)$, and the uniqueness of the maximal entropy measure.  

\section{Statistical Stability}

Now we prove the statistical stability for the  maximizing probability measure $\mu$.
That is, given $f_n \to f$ in the $C^1-$topology,  then $\mu_n \to \mu$ in
weak-$*$ topology, where $\mu_n$ (respectively,  $\mu$)  is the maximizing entropy measure for $f_n$ (respectively, $f$).

Let us fix such $f$, and  consider the  collection  $\cC$ whose elements are open subsets $A \subset M$ whose frontier are $\mu$-zero sets 
with the form $A= \cup_{x \in B} M_x$, for some ball  $B \subset N$  with $\nu$-zero  frontier. Also denote by $\hat \cC \supset \cC$ 
the collection whose elements are  nonnegative interate of some element of $\cC$.
 Observe that, if we fix $k \in \natural$, $ f^k(\cup_{x \in N} M_x)$ is a neighborhood 
for the attractors $\Lambda_n$ where $\mu_n$ are supported,  for all sufficiently big $n$.  
Note that $\hat \cC$ is a neighborhood basis for $\Lambda$.

The key ingredient for the proof is the lemma:

\begin{lemma}
Let  $\hat A \in \hat\cC$. Then $\mu_n(\hat{A}) \to \mu(\hat{A})$ as $n \to +\infty$.
\end{lemma}
\begin{proof} 

Given $\hat A= f^k(A)$, with  $A= \cup_{x \in B} M_x$. We start with the case  $k= 0$, that is, first we prove that 
$\mu_n(A) \to \mu(A)$ as $n \to +\infty$. 

Set $A_n:=\Pi^{-1}_n(B)$. Therefore, $\mu_n(A_n)=\nu_n(B)$, where $\nu_n$ is the maximizing measure $g_n$ as in \cite{CV13}.
We also have  $\mu(A)=\nu(B)$, where $\nu$ is the entropy maximizing probability associated to  $g$, as in  \cite{CV13}.

Given  $\epsilon> 0$, take $B^+ \supset B \supset B^-$,  $\nu-$zero frontier such that
$$
\nu(B^+)- \epsilon/3< \nu(B)< \nu(B^-)+ \epsilon/3, 
$$
Let us also assume that
$A_n^\pm :=\Pi_n^{-1}(B^\pm)$, 
with $\mu-$zero frontier such that there exists  $n_2$ 
that forall $n \geq n_2$ 
${A_n^+} \supset A \supset {A_n^-}$ 
and 
$$
\mu(A_n^+)- \epsilon/3< \mu(A)< \mu(A_n^-)+ \epsilon/3, 
$$
hold.

Such sets exist by the $C^0-$convergence of  (strong stable/center-unstable) foliations  for $f_n$ to the respective foliations for $f$.

On the one hand,  $\exists n_1 \geq n_2$ such that
$$
\mu( A)- \mu_n(A) \leq \mu(A)- \mu_n(A) \leq \mu(A)- \mu_n(A_n^-)= \nu(B)- \nu_n(B^-) \leq \frac{2  \epsilon}{3}, 
$$
for all $n \geq n_1$, as $\nu_n(B^-) \to \nu(B^-)$ by the statistical stability for  
$g$ proved in  \cite{CV13}.

In the same manner, we prove the other inequality, implying there exists $n_0 \geq n_1$ such that
$$
|\mu(A)- \mu_n(A)|< \epsilon, \forall n \geq n_0.
$$

The same arguments also are valid for the case $k> 0$.

This finishes the lemma. 
 
\end{proof}

\begin{theorem}
Given $\varphi: M \to \re$ a continuous function,  then $\ds\int_M \varphi d\mu_n \to \int_M \varphi d\mu$.
\end{theorem}
\begin{proof}

Let $\epsilon> 0$ given, and the $\delta> 0$ we obtain by the uniform continuity of  $\varphi$ 
associated to $\epsilon/9$. Take a covering $\ds\cup_{j= 1}^k C_j$,  $C_j \in \cC$ de  
$\Lambda$, with diameter less then $\delta/3$. There is also  $n_0$ such that $\ds\cup_{j= 1}^k C_j \supset \Lambda_n$,
$\forall n \geq n_0$. In particular, $\mu_n(M \setminus \ds\cup_{j= 1}^k C_j)= 0$, $\forall n \geq n_0$.

Consider a partition of unity  $\{\psi_j, j= 1, \dots, k\}$ associated to $\ds\cup_{j= 1}^k C_j$.

For each  $C_j$, take $x_j \in C_j$ and set
$
\hat \varphi:= \sum_{j= 1}^k \varphi(x_j) \psi_j.
$

Therefore, $\|\varphi - \hat\varphi\|_\infty< \epsilon/3$.

Now, take $n_1 \geq n_0$ such that 
$$
|(\mu_n- \mu)(C_j)|< \frac{\epsilon}{3k \|\varphi\|_\infty}, \forall n \geq n_1.
$$

So, we conclude that 
$$
\modulo{\int_M \varphi d\mu_n - \int_M \varphi d\mu} \leq
\modulo{\int_M \varphi d\mu_n - \int_M \hat\varphi d\mu_n} + 
\modulo{\int_M \hat \varphi d\mu_n - \int_M \hat \varphi d\mu}+ 
$$
$$
\modulo{\int_M \varphi d\mu - \int_M \hat\varphi d\mu} \qquad
$$
$$
\leq \|\varphi- \hat \varphi\|_\infty + \sum_{j= 1}^k \|\varphi\|_\infty |\mu_n(C_j)- \mu(C_j)|+ \|\varphi- \hat \varphi\|_\infty < \epsilon, 
\forall n \geq n_0. 
$$

\end{proof}

\section{Cones and Projective Metrics}\label{s.cones}

We recall here some necessary results in Projective Metrics defined in Cones whose proofs can be found in [Li95,Ba00,Vi95].

Given a linear space $E$ we say that $C \subset E \backslash \{0\}$ is a convex cone if 
$$
t>0 \textrm{ and } v \in C \Rightarrow t \cdot v \in C.
$$
and
$$
t_1, t_2 > 0 \text{ and } v_1, v_2 \in C \Rightarrow t_1 \cdot v_1+ t_2 \cdot v_2 \in C.
$$

We define  $\ov C$ to be the set of points $w \in E$ such that there exists $v \in C$ and a sequence of positive numbers $\l(t_n\r)_{n \in \N}$, going to zero, such that $w+t_n\cdot v \in C$ forall  $n \in \N$. We will only consider the so called {\em projective cones}, such that  
$$
\fecho{C} \cap \l(-\fecho{C}\r)=\{0\}.
$$

We then define
$$
\alpha_C(v,w)=\sup\l\{t>0; w-t\cdot v \in C\r\}
$$
and
$$
\beta_C(v,w)=\inf\l\{s>0; s\cdot v-w \in C\r\}.
$$

We convention $\sup \emptyset=0$ and $\inf \emptyset=+\infty$. The projective metrics associated to $C$ is given by
$$
\theta(v,w)=\log\f{\beta_C(v,w)}{\alpha_C(v,w)}.
$$

Indeed, 
\begin{proposition} Given a projective cone $C$ then $\theta(\cdot,\cdot): \fecho{C} \times \fecho{C} \to [0,+\infty]$ is a metrics in the projective space of  $C$, that is,
\begin{itemize}
	\item $\theta(v,w)=\theta(w,v)$.
	\item $\theta(u,w) \leq \theta(u,v)+\theta(v,w)$.
	\item $\theta(v,w)=0$ iff there exists $t>0$ such that $v=t\cdot w$.
\end{itemize}
\end{proposition}

The proof of the following essential result can be found  in \cite[Proposition~2.3]{Vi97}.

\begin{theorem}\label{t.Birkhoff}
Let $E_1$ and $E_2$ be linear spaces and  let $C_1 \subset E_1$ and $C_2 \subset E_2$ be projective cones. If $L:E_1 \to E_2$ is a linear  operador such that $L( C_1)\subset C_2$ and 
$$D=\sup\l\{\theta_2(L(v),L(w)); v,w \in C_1\r\}<\infty$$
then
\begin{equation*}
\theta_2 (L(v),L(w)) \leq \left( 1 - e^{-D} \right)\theta_1(v,w),
\end{equation*}
for all $v,w\in C_1$.
\end{theorem}

\section{Ruelle-Perron-Frobenius Operator and Invariant Cones}

We recall that the main goal of this work is to deduce good statistical properties of the maximal entropy probability measure associated to the dynamics $f$. The technique presented  use the Ruelle-Perron-Frobenius operator(for simplicity called transfer operator) and its  duality with the Koopman operator, $U(\fhi)=\fhi \circ f$, to obtain the exponential decay of correlations and consequently the central limit theorem.

However, this technique may also be useful to prove exponential decay of correlations and consequently the central limit theorem for more general equilibrium states, 
not just particularly for measures of maximum entropy. We recall that given a map $f: \Lambda \to \Lambda$, and a fixed potential $\phi: \Lambda \to \R$, we say that a measure $\eta$ is an equilibrium state for $f$ with respect to $\phi$ if
$$
h_{\eta}(f)+ \int{\phi}d\eta=\sup\l\{h_{\mu}(f)+\int{\phi d\mu}; \mu \text{ is an $f$-invariant probability}\r\}.
$$
That is, the variational principle tells us that $\eta$ carries out the topological pressure $P(f,\phi)$. The reader can easily see that in the case where the potential $\phi$ is a constant, 
obtain an equilibrium state is equivalent to obtain a maximum entropy measure. What we do in this section is to obtain some preliminar results,  for  more general potentials  than constant potentials, namely,  low variation potentials. That is, we assume that $\sup \phi -\inf \phi < \varepsilon$
for some small enough $\varepsilon$. Moreover such potential must 
belong to the following cone:
\begin{equation}\label{potencial}
\holder{e^\phi} \leq \varepsilon \inf{e^\phi} 
\end{equation}
where $\holder{e^\phi}=\infimo{}{C>0; |e^\phi(x)-e^\phi(y)|\leq Cd(x,y)^{\alpha}, \forall x,y \in \Lambda}$.
Let $E$ be the space of continuous functions $\fhi:\Lambda \seta \R$. Define the Ruelle-Perron-Frobenius operator  $\L:E \seta E$ given
$$\L(\fhi)(y)=\fhi(f^{-1}(y))e^{\phi(f^{-1}(y))}$$
where $\phi$ satisfies the above conditions.

Our inspiration  is the work developed in \cite{CV13}, where the exponential decay of correlations and other good statistical and regularity properties
are proven for the unique equilibrium state in a nonuniformly expanding context. Castro-Varandas defined suitable cones for the
 Ruelle-Perron-Frobenius (or transfer) operator$\L$, proving the  invariance and  the finite diameter for the  image of such cones by $\L$. 

More precisely, the basic cone used by \cite{CV13} is the cone of H\"older continuous, positive functions  $\fhi$ such that 
$\holder{\fhi} \leq \kappa \inf{\fhi}$. The invariance of such cone by $f$ is due some increase in the regularity given by the contraction of some
inverse branch of $f$.
 In our context, however, we always have backward expansion in stable directions for the points into each strong stable manifold $\Pi^{-1}(y)$ instead of contraction. Since for the case of entropy (potential $\phi \equiv 0$) the transfer operator $\L$,  is just the composition of each observable $\fhi$ with $f^{-1}$, it is obvious that  the H\"older constants of $\L(\fhi)$, can not  better, if one take a cone as in \cite{CV13}.
 
In order to avoid this undesirable effect in stable directions, we will  analyse the action of $\L$ in some kind of averages taken in each stable leaf restricted to the attractor  $\La$. We will write the lowercase letter $\gamma$ to denote a  stable leaf (instersected with $\Lambda$) and $\F^s$ will denote the stable foliation.   

Fixed $y \in N$, let $y_j$ such that $g(y_j)=y$, where $j \in\{1, \cdots, deg(g)\}$. Writing $\gm=\Pi^{-1}_{\Lambda}(y)$ and $\gmj=\Pi^{-1}_{\Lambda}(y_j)$, it follows that $f(\gmj)\subset \gm$, since $\Pi \circ f(x)=g \circ \Pi(x)=g(y_j)=y$, $\forall x \in \gmj$.

Let $\textbf{p}$ be the degree of $g$. Let us construct a family of measures $\l\{\mg\r\}_{\gm \in \F^s}$ supported in $\Lambda$, such that for all $\gmc$, where $f^n\l(\gmc\r) \subset \gm$, we have $\mg\l(f^n\l(\gmc\r)\r)=\f{1}{p^n}$. In particular $\mg(\gm)=1$. Furthermore, for all $\gmj$, with $f(\gmj)\subset \gm$ we will obtain

$$\int_{f(\gmj)}\psi \dmg=\f{1}{p}\intgj{\psi\circ f}.$$

The construction of such family of measures is rather natural. Fix $\gm = \Pi_{\Lambda}^{-1}(y)$ and $n \in \N$, $n> 0$. By setting $\gmj:=\Pi_{\Lambda}^{-1}(y_j)$, where $y_j \in g^{-n}(y)$, one can write $\gm = \ds\dot{\bigcup}_{j=1}^{p^n}f^n(\gmj)$, since $f^n$ is a bijection in  $\Lambda$ and $\Pi \circ f^n = g^n \circ \Pi$. Therefore, $\l\{f^n(\gmj)\r\}_{j=1}^{p^n}$ is a sequence of partitions in $\gm$. As $\gmj=\Pi_{\Lambda}^{-1}(y_j)$ and $f^n:M_{y_j} \to M_{g^n(y_j)}$ is a $\lambda_s^n$-contraction it follows that the diameter of $\l\{f^n(\gmj)\r\}_{j=1}^{p^n}$ goes to  zero. So, we just define $\mg$ in the elements of such partition by  mass distribution
$$
\mg(f^n(\gmj))=\f{1}{p^n}
$$

\begin{figure}
\includegraphics[width=10cm]{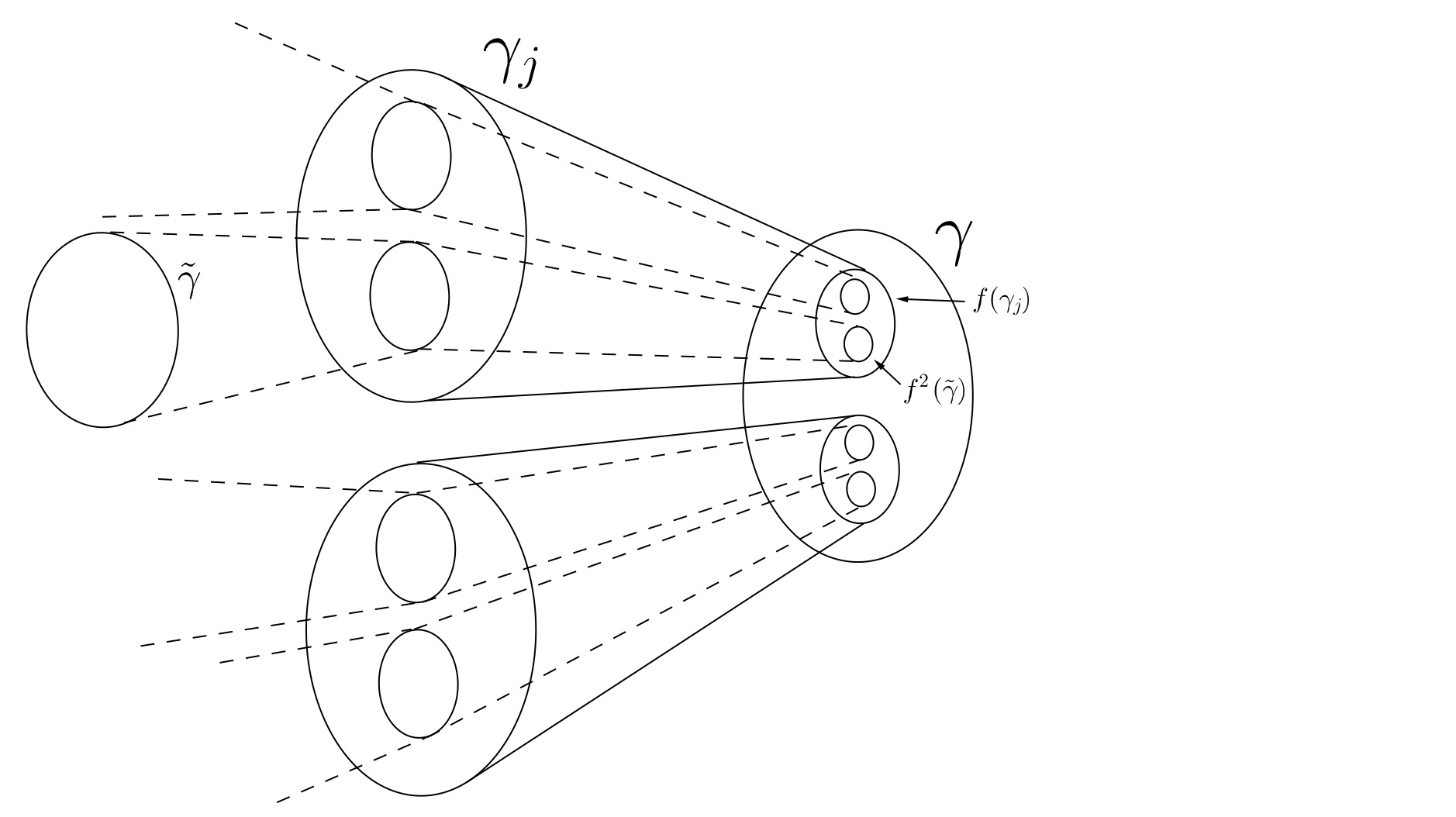}
\caption{Mass distribution}
\label{Distribuicao de massa}
\end{figure}
and extend  $\mg$ by approximation to any Borelian $A \subset \Lambda$. 

If $\gmj=\Pi_{\Lambda}^{-1}(x_j)$,  $x_j \in g^{-1}(x)$, then 
$$
\mg(A)=\mg(A \cap \gm)=\mg\l(A \cap \ds\bigcup_{j=1}^{p}f(\gm_{j})\r)=\mg\l(\ds\bigcup_{j=1}^{p}\l(A \cap f(\gmj)\r)\r)=\ds\sum_{j=1}^{p}\mg(A \cap f(\gmj))
$$
Seting $\mgj(A):=p \cdot \mg(f(A \cap \gm_j))$ we obtain $\mg(A \cap f(\gmj))=\f{1}{p}\mgj(f^{-1}(A))$ and so 
$$
\mg(A)=\f{1}{p}\ds\sum_{j=1}^{p}\mgj(f^{-1}(A)).
$$
We conclude that for any measurable set $A$,  its indicator function $\chi_A$ satisfies
$$
\ds\int_{f(\gmj)}\chi_A\dmg=\f{1}{p}\intgj{\chi_A \circ f}
$$
By Lebesgue Dominated Convergence Theorem,  for any $g:\Lambda \to \R$ continuous we have  
\begin{equation} \label{MudancaDeVariavel}
\ds\int_{f(\gmj)}g\dmg=\f{1}{p}\intgj{g \circ f}.
\end{equation}
Note also that for all  $\gmc$,  $f^n\l(\gmc\r) \subset \gm$, we have $\mg\l(f^n\l(\gmc\r)\r)=\f{1}{p^n}$. So 
it follows that  for all $\gmt$ such that $f^n\l(\gmt\r) \subset \gmj$ and  $f\l(\gmj\r) \subset \gm$
$$
\mgj\l(f^n\l(\gmt\r)\r)= \,p\mg\l(f\l(f^n\l(\gmt\r) \cap \gmj\r)\r)=\, p \mg\l(f^{n+1}(\gmt)\r)= \,  \f{p}{p^{n+1}}= \, \f{1}{p^n}
$$
holds. 

That is, $\mgj$ is the mass distribution measure constructed for $\gmj$.

Moreover, for $y \in N$  and $y_j$ such that $g(y_j)=y$, $j \in\{1, \cdots, p\}$ if we consider  $\gm=\Pi^{-1}_{\Lambda}(y)$ and $\gmj=\Pi^{-1}_{\Lambda}(y_j)$,  $f(\gmj)\subset \gm$, then $\gm=\dot{\bigcup}_{j=1}^{p}f(\gmj)$. Therefore, for all measurable bounded 
 function $\psi:\gm \to \R$ it follows that
$$
\int_{\gm} \psi d\mu_\gm=\sumjap\int_{f(\gm_j)}\psi d\mu_\gm.
$$
For $\rho:\gm \seta \R$, we conclude that 
$$
\intg{\L(\fhi)\rho}
=\sumjap\int_{f(\gm_j)}\L(\fhi)\rho \dmg
=\sumjap\f{1}{p}\intgj{\fhi \cdot e^\phi \cdot \rho\circ f}
$$
defining $\rj:=\f{1}{p}\rho\circ f e^{\phi}$, we have
$$
\intg{\L(\fhi)\rho} =\sumjap \intgj{\fhi\rj}.
$$
We will study the action of the transfer operator in the strong stable leaves  via its action on the integrals of densities in a suitable cones of functions which are defined in each strong stable leaf.
More precisely, for each $\gm \in \mathcal{F}^s$ we define the auxiliary cone of H\"older continuous functions
$$
\D(\gm,\kappa):= \{ \rho:\gm \seta \rho > 0 \text{ and } |\rho|_{\alpha}< \kappa \inf{\rho}\},
$$
with $|\rho|_{\alpha}=\infimo{}{C>0; |\rho(x)-\rho(y)|\leq Cd(x,y)^{\alpha}, \forall x,y \in \gm} $. 

Note that for  $\rho$  in a cone  $\D(\gm,\kappa)$ we have  $\sup{\rho} \leq \inf{\rho\l(1+\kappa \cdot diamM^\alpha\r)}$. 

The next lemma is about the invariance of the auxiliary cones under the action of the transfer operator. 

\begin{lemma}\label{invconeaux}
There exist sufficiently small  $0<\lambda<1$  and $\kappa>0$, such that the following itens hold:
\begin{enumerate}

\item If $\rho \in \D(\gm,\kappa)$ then $\rho_j \in \D(\gm_j,\lambda\kappa)$ for all  $j \in\{1, \ldots,p\}$. 
\item For all $\gm \in \F_{loc}^s$,  if $\rho,\rc \in \D(\gm,\lambda\kappa)$ then $\theta(\rho,\rc)\leq 2\log\l(\f{1+\lambda}{1-\lambda}\r)$. 
\item If $\rl,\rll \in \D(\gm,\kappa)$ then there exists $\Lambda_1 =1-\l(\f{1-\lambda}{1+\lambda}\r)^2$ such that  $\theta_j(\rl_j,\rll_j)\leq \Lambda_1\theta(\rl,\rll)$ for all $j \in\{1, \ldots,p\}$;

\end{enumerate}
where $\theta_j$ and $\theta$ are  the projective metrics associated to $\D(\gmj,\kappa)$ and $\D(\gm,\kappa)$, respectively.
\end{lemma}

\begin{proof}

(1) In our context we suppose $\sup\phi-\inf\phi < \varepsilon$ and $\holder{e^\phi}<\varepsilon \inf e^\phi$. Therefore
$$
\begin{array}{rcl}
\f{|\rho_j|_\alpha}{\inf\l\{\rho_j\r\}}&=&\f{\l|\f{1}{p}\rho\circ f \cdot e^\phi\r|_\alpha}{\inf\l\{\f{1}{p}\rho\circ f \cdot e^\phi\r\}} =\,\f{|\rho\circ f \cdot e^\phi|_\alpha}{\inf\l\{\rho\circ f \cdot e^\phi\r\}}\\
& \leq & \f{|\rho \circ f|_\alpha \cdot e^{\sup \phi} + \sup \l\{\rho \circ f\r\} \cdot |e^\phi|_\alpha}{\inf \rho \cdot e^{\inf \phi}}\\
& \leq & \f{\lambda_s^{\alpha} \kappa \inf \rho \cdot e^{\sup \phi}}{\inf \rho \cdot e^{\inf \phi}} + \f{(1+\kappa \cdot diamM^\alpha)\inf \rho \cdot |e^\phi|_\alpha}{\inf \rho \cdot e^{\inf \phi}}\\
& \leq & \lambda_s^{\alpha}\kappa e^\varepsilon + (1 + \kappa \cdot diam  M^\alpha) \varepsilon  = (\lambda_s^{\alpha} e^\varepsilon + diam  M^\alpha \epsilon)\kappa+\varepsilon\\
\end{array}
$$
In order to guarantee a $0<\lambda <1$ such that 
$$
(\lambda_s^{\alpha} e^\varepsilon + diam  M^\alpha \varepsilon)\kappa+\varepsilon < \lambda \kappa
$$
it is sufficient to obtain 
$$
\f{(\lambda_s^{\alpha} e^\varepsilon + diamM^\alpha \varepsilon)\kappa+\varepsilon}{\kappa} < \lambda <1.
$$
For that we just need 
$$
\f{(\lambda_s^{\alpha} e^\varepsilon + diamM^\alpha \varepsilon)\kappa+\varepsilon}{\kappa} < 1 
$$
or, equivalently, 
\begin{equation} \label{PrimeiraCond}
\kappa > \f{\varepsilon}{1-(\lambda_s^{\alpha} e^\varepsilon+diamM^\alpha\varepsilon)}.
\end{equation}
Note that  $\lambda$ and $\kappa$ can be  chosen in order to satisfy the above equation 
since we  choose in our hipothesis $\varepsilon>0$ and $0<\lambda_s<1$
suitably small.

(2) By a triangular argument, it is  sufficient to bound $\theta(1,\rho)$ for $\rho \in \D(\gm,\lambda\kappa)$.  There is no loss of generality in assuming that  $\inf \rho=1$.
So, for  $t=1-\lambda$ we have
$$
\f{\holder{\rho-t}}{\inf{(\rho-t)}}=\f{\holder{\rho}}{\inf{\rho}-t}<\f{\lambda\kappa}{1-t}=\f{\lambda\kappa}{\lambda}=\kappa.
$$
Since $\inf \rho=1$ it follows that $\rho-t \geq \inf \rho-(1-\lambda)=\lambda > 0$ which guarantees $\alpha(1,\rho) \geq 1-\lambda$. 
On the other hand, by setting $s=1+\lambda$ we obtain
$$
\f{\holder{s-\rho}}{\inf{(s-\rho)}}=\f{\holder{\rho}}{s-\inf{\rho}}<\f{\lambda\kappa}{s-1}=\f{\lambda\kappa}{\lambda}=\kappa.
$$
As $\sup{\rho} \leq \inf{\rho}\l(1+\kappa diamM^{\alpha}\r)=1+\kappa diamM^{\alpha}$ taking $\kappa$ such that $\lambda> \kappa diamM^{\alpha}$, it follows that $s-\rho=1+\lambda-\rho \geq 1+\lambda-\sup\rho \geq 1+\lambda - \l(1+\kappa diamM^{\alpha}\r)>0$. Therefore $\beta(1,\rho) \leq 1+\lambda$. So we conclude that $\theta(\rho,\rc)\leq 2\log\l(\f{1+\lambda}{1-\lambda}\r)$.

Finally, in order to prove  (3) it is sufficient to note that  by item $(1)$ we have  $\rho_j \in \D(\gm_j,\lambda\kappa)$ for all $j \in\{1, \ldots,p\}$ and by item $(2)$ the diameter  $\D(\gm_j,\lambda\kappa)$ in $\D(\gm_j,\kappa)$ is, at most, $2\log\l(\f{1+\lambda}{1-\lambda}\r)$. Therefore, the result goes on by theorem \ref{t.Birkhoff}, considering $\Delta=2\log\l(\f{1+\lambda}{1-\lambda}\r)$ and the linear map
$$
\begin{array}{rcl}
\rho \mapsto \f{1}{p}\rho \circ f e^{\phi}
\end{array}
$$
we have $\theta_j(\rl_j,\rll_j)\leq \Lambda_1\theta(\rl,\rll)$ where
$$
\Lambda_1=1-\ds e^{-\Delta}=1-\l(\f{1-\lambda}{1+\lambda}\r)^2
$$
\end{proof} 
For the definition of the main cone on which we will apply the transfer operator we need to define a notion of distance
between two strong stable  leaves $\gm$ and $\gmt$ in $\F^s$. Given $x,y \in N$ let $\gm=\Pi_{\Lambda}^{-1}(x)$ and $\gmt=\Pi_{\Lambda}^{-1}(y)$. 
Suppose $\pi=\pi_{x,y}: \gmt  \to  \gm$ satisfies 
$$
\begin{array}{rclc}
\intg{\fhi}=\intgt{\fhi\circ\pi}
\end{array}
$$
for all continuous function  $\fhi$ and define the distance 
$d(\gm,\gmt)=\sup\l\{d(\pi(p),p); p \in \gmt \r\}$. 

Now let us define our main cone. Denote by $\Dnorm{\gamma}$ the set of densities
 $\rho \in \D(\gm,\kappa)$ such that $\intg{\rho}=1$.
Given $b>0$, $c>0$ and $\kappa$ as in lemma $\ref{invconeaux}$, let $C(b,c,\alpha)$ be the cone of functions $\fhi \in E$ satisfying for all $\gm \in \F^s$ the following:

\begin{description}

\item[(A)] For all $\rho \in \D(\gm,\kappa)$:
$$ \intg{\fhi\rho} >0$$
\item [(B)] For all $\rl, \rll \in \Dnorm{\gamma}$:

$$\modulo{\intg{\fhi\rl}- \intg{\fhi \rll}} < b \theta\l(\rl,\rll\r) \infimo{\rho \in \Dnorm{\gamma}}{\intg{\fhi\rho}}$$

\item [(C)] Given any $\gmt$ sufficiently close to  $\gm$:

$$\modulo{\intg{\fhi} - \intgt{\fhi}} < c d(\gm,\gmt)^{\alpha}\infimo{\gm}{\intg{\fhi}} $$

\end{description}

The proof of the next Lemma follows from standard arguments.

\begin{lemma}
$C\l(b,c,\alpha\r)$ is a projective cone.
\end{lemma}

Now, we have:

\begin{proposition} \label{InvarianciaDoCone}
Let $\phi$ be constant. There exists $0<\sigma<1$ such that $\L(C(b,c,\alpha)) \subset C(\sigma b,\sigma c,\alpha)$ for sufficiently large $b$,$c > 0$.
\end{proposition}
\begin{proof} 

Invariance of condition  $(A)$: 
Let $\fhi \in C(b,c,\alpha)$. We know that $\intg{\L(\fhi)\rho}  = \sumjap \intgj{\fhi\rj}$ and by lemma \ref{invconeaux} $\rj \in \D(\gmj,\kappa)$. Therefore, $\intg{\L(\fhi)\rho}>0$.

Invariance of condition $(B)$: 
Denoting $\f{\rj}{\intgj{\rj}}$ by $\rc_j$ we can write
$$
\begin{array}{rcl}
\infimo{\rho \in \Dnorm{\gm}}{\intg{\L(\fhi)\rho}} & \geq & \sumjap \infimo{\rho \in \Dnorm{\gm}}{\intgj{\fhi\rj}}\\
&= & \sumjap \infimo{\rho \in \Dnorm{\gm}}{\intgj{\fhi\rc_j}\intgj{\rj}}\\
&\geq & \sumjap \infimo{\rho \in \Dnorm{\gmj}}{\intgj{\fhi\rho}}\infimo{\rho \in \Dnorm{\gm}}{\intgj{\rj}}\\
\end{array}
$$
Given $\rl,\rll \in \Dnorm{\gm}$ writing $\rlj/\intgj{\rlj}$ and $\rllj/\intgj{\rllj}$ for $\rbj$ and $\rbbj$, respectively, follows that
\begin{equation} \label{PrimeiraDesConB}
\begin{array}{rcl}
\modulo{\intg{\L (\fhi) \rl} -\intg{\L(\fhi) \rll}} & \leq & \sumjap \modulo{ \intgj{\fhi\rbj}- \intgj{\fhi\rbbj}} \intgj{\rlj}\\
&+&\sumjap \intgj{\fhi\rbbj}\modulo{\intgj{\rlj}-\intgj{\rllj}}.\\
\end{array}
\end{equation}
By hypothesis,  $\fhi$ is in the cone and by lemma \ref{invconeaux}, we have

\begin{equation} \label{desphidensidades}
\begin{array}{rcl}

\modulo{ \intgj{\fhi\rbj}- \intgj{\fhi\rbbj}} & \leq & b\theta_j\l(\rbj,\rbbj\r) \infimo{\rho \in \Dnorm{\gmj}}{\intgj{\fhi\rho}}\\	
& \leq & b\Lambda_1\theta\l(\rl,\rll\r) \infimo{\rho \in \Dnorm{\gmj}}{\intgj{\fhi\rho}}.\\
\end{array}
\end{equation}

For all $\rc \in \Dnorm{\gm}$ we obtain the following estimative

\begin{equation}\label{DesDasDensidades}
\f{\intgj{\rcj}}{\infimo{\rho \in \Dnorm{\gm}}{\intgj{\rj}}} \leq \l(1+\kappa diamM^\alpha\r)^2
\end{equation}

In fact, given $\delta>0$ there exists $\rt \in \Dnorm{\gm}$ such that $\intgj{\rtj} \leq (1+\delta)\infimo{\rho \in \Dnorm{\gm}}{\intgj{\rj}}$.
Moreover, as $\rc$ and $\rt$ are normalized, we necessarily have $\inf{\rc} \leq 1$ and $\sup{\rt} \geq 1$. Therefore,
$$
\f{\rcj}{\rtj}=\f{\f{1}{p}\rc \circ fe^{\phi}}{\f{1}{p}\rt \circ fe^{\phi}} \leq \f{\sup \rc}{\inf \rt}\leq \f{\l(1+\kappa diamM^\alpha\r)\inf\rc}{\l(1+\kappa diamM^\alpha\r)^{-1}\sup\rt}=\l(1+\kappa diamM^\alpha\r)^2.
$$

And so $\intgj{\rcj} \leq \l(1+\kappa diamM^\alpha\r)^2 \intgj{\rtj}$, we obtain for all 
 $\delta >0$
$$
\f{\intgj{\rcj}}{\infimo{\rho \in \Dnorm{\gm}}{\intgj{\rj}}} \leq \f{(1+\delta)\l(1+\kappa diamM^\alpha\r)^2\intgj{\rtj}}{\intgj{\rtj}} 
$$
$$
 \qquad \quad \qquad \leq (1+\delta)\l(1+\kappa diamM^\alpha\r)^2 
$$
giving the estimative we wish.

Now,  for fixed $j$, we obtain
\begin{equation}
\begin{array}{rl} \label{estimativa_parcela_1}
\f{\modulo{ \intgj{\fhi\rbj}- \intgj{\fhi\rbbj}} \intgj{\rlj}}{\infimo{\rho \in \Dnorm{\gmj}}{\intgj{\fhi\rho}}\infimo{\rho \in \Dnorm{\gm}}{\intgj{\rj}} \theta(\rl,\rll)}  \leq &
\f{b\Lambda_1\theta\l(\rl,\rll\r)\intgj{\rlj}}{\infimo{\rho \in \Dnorm{\gm}}{\intgj{\rj}} \theta(\rl,\rll)}\\
&\\
 \leq & \l(1+\kappa diamM^\alpha\r)^2\Lambda_1 b\\
\end{array}
\end{equation}

Let us analyse the second parcel of $\ref{PrimeiraDesConB}$. First, note that for all  $\rc \in \Dnorm{\gm}$, denoting  $\rcj/\intgj{\rcj}$ by $\bar{\rc}_j$, we claim that
$$
\f{\intgj{\fhi\bar{\rc}_j}}{\infimo{\rho \in \Dnorm{\gm}}{\intgj{\fhi \rbj}}} < b\log\l(\f{1+\lambda}{1-\lambda}\r)^2+1
$$
In fact, analogously to what was done in  $\ref{DesDasDensidades}$, it is sufficient to
 to note that, since $\fhi$ is in the cone, we have
$$
\f{\intgj{\fhi\bar{\rc}_j}}{\intgj{\fhi \rbj}} < b\theta(\bar{\rc}_j,\rbj)+1=b\theta(\rcj,\rj)+1
$$
By \ref{invconeaux}, we conclude the proof  of our claim. 

Now, we stablish the other necessary estimative:
$$
\f{\modulo{\intgj{\rlj}-\intgj{\rllj}}}{\theta\l(\rl,\rll \r)\infimo{\rho \in \Dnorm{\gm}}{\intgj{\rj}}}
 \leq
 2\l( 1+\kappa diamM^\alpha\r)^2.
$$
In order to prove this last estimative we observe that
$$
\f{\rlj}{\rllj}\leq\f{\sup{\rl}}{\inf{\rll}}\leq\f{\sup{\rl}/\inf{\rl}}{\inf{\rll}/\sup{\rll}}=e^{\theta_+\l(\rl,\rll \r)}\leq e^{\theta\l(\rl,\rll \r)}
$$
Therefore, by assuming without loss of generality that  $\intgj{\rlj} \geq \intgj{\rllj}$ we obtain
$$
\f{\modulo{\intgj{\rlj}-\intgj{\rllj}}}{\theta\l(\rl,\rll \r)\infimo{\rho \in \Dnorm{\gm}}{\intgj{\rj}}}
\leq
\f{\ds\l(e^{\ds\theta\l(\rl,\rll \r)}-1\r)\intgj{\rllj}}{\theta\l(\rl,\rll \r)\infimo{\rho \in \Dnorm{\gm}}{\intgj{\rj}}}
$$
for $\theta\l(\rl,\rll \r) \leq 1$ it follows $\f{e^{\ds\theta\l(\rl,\rll \r)}-1}{\theta\l(\rl,\rll \r)} < 2$ and so we
obtain our estimative.

If $\theta\l(\rl,\rll \r) \geq 1$ we also have that
$$
\f{\modulo{\intgj{\rlj}-\intgj{\rllj}}}{\theta\l(\rl,\rll \r)\infimo{\rho \in \Dnorm{\gm}}{\intgj{\rj}}}
\leq
\f{\modulo{\intgj{\rlj}-\intgj{\rllj}}}{\infimo{\rho \in \Dnorm{\gm}}{\intgj{\rj}}}
$$
$$
\qquad \quad \qquad \qquad \qquad \qquad \leq 2\l( 1+\kappa diamM^\alpha\r)^2
$$
and again for fixed $j$ and by  writing $M(\kappa,\alpha)$ for  $\l( 1+\kappa diamM^\alpha\r)^2$,
\begin{equation}
\begin{array}{rl} \label{estimativa_parcela_2}
\f{\intgj{\fhi\rbbj}\modulo{\intgj{\rlj}-\intgj{\rllj}}}{\infimo{\rho \in \Dnorm{\gm}}{\intgj{\fhi\rj}}\infimo{\rho \in \Dnorm{\gm}}{\intgj{\rj}} \theta(\rl,\rll)}
 \leq &
\f{\intgj{\fhi\rbbj}}{\infimo{\rho \in \Dnorm{\gm}}{\intgj{\fhi\rj}}}2M(\kappa,\alpha)\\
&\\
 \leq  \l( b\log\l(\f{1+\lambda}{1-\lambda}\r)^2+1\r)2M(\kappa,\alpha)
 \leq & 2M(\kappa,\alpha)\log\l(\f{1+\lambda}{1-\lambda}\r)^2b + 2M(\kappa,\alpha)\\
\end{array}
\end{equation}

The inequalities \ref{estimativa_parcela_1} and \ref{estimativa_parcela_2} does  not depend on $j$, so
$$
\begin{array}{rcl}
\f{\modulo{\intg{\L (\fhi) \rl} -\intg{\L(\fhi) \rll}}}{\infimo{\rho \in \Dnorm{\gm}}{\intg{\L(\fhi)\rho}} \theta(\rl,\rll)}
& \leq &
M(\kappa,\alpha)\Lambda_1 b + 2M(\kappa,\alpha)\log\l(\f{1+\lambda}{1-\lambda}\r)^2b + 2M(\kappa,\alpha)\\
& = &
\l(\Lambda_1 + 2\log\l(\f{1+\lambda}{1-\lambda}\r)^2\r)M(\kappa,\alpha)b + 2M(\kappa,\alpha)\\
\end{array}
$$
We need that the  term which  multiplies b above to be less than 1. Recall that by lemma (\ref{invconeaux}), $\Lambda_1=1-\l(\f{1-\lambda}{1+\lambda}\r)^2$. So, 
we need to guarantee that 
$$\l(1-\l(\f{1-\lambda}{1+\lambda}\r)^2+2\log\l(\f{1+\lambda}{1-\lambda}\r)^2\r)\l(1+\kappa diamM^\alpha\r)^2<1$$
Also by lemma (\ref{invconeaux}),  we can choose $\kappa$, such that $\kappa diamM^\alpha < \lambda$, $\lambda$ to be fixed. So let us find a bound $0<\lambda<1$ such that 
$$\l(1-\l(\f{1-\lambda}{1+\lambda}\r)^2+2\log\l(\f{1+\lambda}{1-\lambda}\r)^2\r)\l(1+\lambda\r)^2<1.$$
It is possible because (\ref{invconeaux}), $0<\lambda<1$ can be taken sufficiently small depending  on the contraction rate in the strong stable directions. So, there exists $0<\tilde{\sigma}_1<1$ 
such that
$$\l(\Lambda_1 + 2\log\l(\f{1+\lambda}{1-\lambda}\r)^2\r)M(\kappa,\alpha)<\tilde{\sigma}_1.$$
Since $M(\kappa,\alpha)$ does not depend on  $b$, for sufficiently large $b$ we can obtain  $\sigma_1<1$ such that
$$
\begin{array}{rcl}
\f{\modulo{\intg{\L (\fhi) \rl} -\intg{\L(\fhi) \rll}}}{\infimo{\rho \in \Dnorm{\gm}}{\intg{\L(\fhi)\rho}} \theta(\rl,\rll)}
& \leq &
\sigma_1 b\\
\end{array}
$$
This prove the strict invariance of condition $(B)$.

Invariance of condition $(C)$: This is where we need that $\phi$ is constant.
We have that
$$\ds\infimo{\gm}{\intg{\L \fhi}} \geq e^{\phi}\ds\infimo{\gm}{\intg{\fhi}}.$$
For $g$ as in \ref{CV}, every  $y \in N$ has at least one pre-image out of the region $\Omega$. So, for 
$\gm=\Pi^{-1}(y)$ and $\gmt=\Pi^{-1}(\tilde{y})$ sufficiently close to $\gm$ such that  $\tilde{y}_i$ is pre-image  $\tilde{y}$, close to  $y_i$, stay in $U_{y_i}$ we obtain that $d(\gm_i,\gmt_i)\leq \lambda_u d(\gm,\gmt)$. 

In fact, let  $x \in \gmt_i$ realizing the distance $d(\gm_i,\gmt_i)$. By a slight abuse 
of notation,  we write  $d$ for the product distance equivalent to the original metrics.
So,  
$$
\begin{array}{rcl}
d(\gm_i,\gmt_i)& = & d(x,\pi_{\tilde
{y}_i,y_i}(x)) =d(\tilde{y}_i,y_i) \leq \lambda_u d(g(\tilde{y}_i),g(y_i))\\
&=&\lambda_u d(\tilde{y},y)=\lambda_u\l[d(\tilde{y},y)+d(\pi_{\tilde{y},y}(f(x)),\pi_{\tilde{y},y}(f(x)))\r]\\
& = &  \lambda_u d(f(x),\pi_{\tilde{y},y}(f(x))) \leq \lambda_u d(\gmt,\gm)
\end{array}
$$
Analogously, in the other cases  we have $d(\gm_j,\gmt_j)\leq L d(\gm,\gmt)$. Furthermore, we can assume with no loss of generality, that $d(\gm_1,\gmt_1)\leq\tilde{\lambda}_u d(\gm,\gmt)$, and for other pre-images we have $d(\gm_j,\gmt_j)\leq \tilde{L} d(\gm,\gmt)$. Note that condition C could not hold for such other pre-images with the same $c$, since they are more distant than the initial leaves $\gm, \gmt$. 
However, just as in Lemma 3.5 of \cite{CV13}, It holds with constant $c(1+ (\tilde L-1)^{\alpha})$ instead of $c$.
For It follows that
$$
\begin{array}{rcl}
\ds\l|\int_{\gm} \L \fhi \dmg - \int_{\gmt} \L \fhi \dmgt \r|
& \leq &
\ds\f{e^{\phi}}{p}\sumjap\l|\int_{\gmj} \fhi \dmgj - \int_{\gmtj} \fhi \dmgtj \r| \\
& \leq &
\f{e^{\phi}c}{p}\infimo{\gm}{\intg{\fhi}} \sumjap d(\gmj,\gmtj)^{\alpha}\\
& \leq &
\f{\tilde{\lambda}_u^{\alpha}+(p-1)(1+ (\tilde L- 1)^\alpha )\tilde{L}^{\alpha}}{p} c d(\gm,\gmt)^{\alpha}\ds\inf_{\gm}\l\{\int_{\gm}\L\fhi d\mu_\gm \r\}\\
\end{array}
$$
We should obtain $\f{\tilde{\lambda}_u^{\alpha}+(p-1)(1+ (\tilde L- 1)^\alpha )\tilde{L}^{\alpha}}{p} < 1$. This is equivalent to
$
(1+ (\tilde L- 1)^\alpha )\tilde{L}^\alpha < \f{p-\tilde{\lambda}_u^\alpha}{p-1}.
$
Due to the  fact that $\tilde{L}\geq1$, we have 
$
\f{p-\tilde{\lambda}_u^\alpha}{p-1} \geq 1
$,
because  $\tilde{\lambda}_u<1$. Therefore, there exists $0<\sigma_2 <1$ such that
\begin{equation}
\modulo{\intg{\L\fhi} - \intgt{\L\fhi}} < \sigma_2 c d(\gm,\gmt)^{\alpha}\infimo{\gm}{\intg{\L\fhi}}  \label{eqC}
\end{equation}
which proves (C).

By setting $\sigma=\max\{\sigma_1,\sigma_2\}$, we finish the proof of the proposition.
\end{proof}

\section{Finite Diameter of the Main Cone}

From now on, up to the end of our text, $\phi$ will be always constant.
In this section, we prove the strict invariance of the main  cone  $C\l(b,c,\alpha\r)$ by the  Ruelle-Perron-Frobenius operator $\L$. First, let us calculate 
the projective metrics $\Theta$. Recall that $\alpha_C(\fhi,\psi)=\sup\l\{t>0; \psi-t\fhi \in C\l(b,c,\alpha\r) \r\}$. 
By (A), for all $\gm \in \F^s_{loc}$ and $\rho \in \D\l(\gm\r)$ we have $\intg{(\psi-t\fhi)\rho}>0$,  that is,
$t < \f{\intg{\psi\rho}}{\intg{\fhi\rho}}.$
By condition (B), one obtains
$$
\modulo{\intg{(\psi-t\fhi)\rl}- \intg{(\psi-t\fhi) \rll}} < b \theta\l(\rl,\rll\r) \infimo{\rho \in \Dnorm{\gamma}}{\intg{(\psi-t\fhi)\rho}}
$$
and so, for all $\rl, \rll$, and $\rc$ in $\Dnorm{\gamma}$ we have
$$
t < \f{\intg{\psi\rl}-\intg{\psi\rll}+b\theta(\rl,\rll)\intg{\psi\rc}}{\intg{\fhi\rl}-\intg{\fhi\rll}+b\theta(\rl,\rll)\intg{\fhi\rc}}
$$
and
$$
t < \f{\intg{\psi\rll}-\intg{\psi\rl}+b\theta(\rl,\rll)\intg{\psi\rc}}{\intg{\fhi\rll}-\intg{\fhi\rl}+b\theta(\rl,\rll)\intg{\fhi\rc}}.
$$
By condition (C),
$$
\modulo{\intg{(\psi-t\fhi)} - \intgt{(\psi-t\fhi)}} < c d(\gm,\gmt)^{\alpha}\infimo{\gm}{\intg{\l(\psi-t\fhi\r)}}
$$
therefore, for all $\gm,\gmc \in \F^s_{loc}$ and  $\gmt$ sufficiently close to $\gm$ we have
$$
t < \f{\intgt{\psi}-\intg{\psi}+cd(\gm,\gmt)\intgc{\psi}}{\intgt{\fhi}-\intg{\fhi}+cd(\gm,\gmt)\intgc{\fhi}}
$$
and
$$
t < \f{\intg{\psi}-\intgt{\psi}+cd(\gm,\gmt)\intgc{\psi}}{\intg{\fhi}-\intgt{\fhi}+cd(\gm,\gmt)\intgc{\fhi}}.
$$
By defining
$$\xi(\gm,\rl,\rll,\rc,\fhi,\psi)
=
\f{\l(\intg{\psi\rll}-\intg{\psi\rl}\r)/\intg{\psi\rc}+b\theta(\rl,\rll)}{\l(\intg{\fhi\rll}-\intg{\fhi\rl}\r)/\intg{\fhi\rc}+b\theta(\rl,\rll)}
$$
and
$$\eta(\gm,\gmt,\gmc,\rc,\fhi,\psi)
=
\f{\l(\intg{\psi}-\intgt{\psi}\r)/\intgc{\psi}+cd(\gm,\gmt)}{\l(\intg{\fhi}-\intgt{\fhi}\r)/\intgc{\fhi}+cd(\gm,\gmt)}.
$$
we can write
$$
\alpha_C(\fhi,\psi)=\inf\l\{\f{\intg{\psi\rho}}{\intg{\fhi\rho}},\f{\intg{\psi\rc}}{\intg{\fhi\rc}}
\xi(\gm,\rl,\rll,\rc,\fhi,\psi),\f{\intgc{\psi}}{\intgc{\fhi}}\eta(\gm,\gmt,\gmc,\rc,\fhi,\psi)\r\}
$$
as $\beta_C(\fhi,\psi)=\alpha_C(\psi,\fhi)^{-1}$ we obtain
$$
\beta_C(\fhi,\psi)=\sup\l\{\f{\intg{\fhi\rho}}{\intg{\psi\rho}},\f{\intg{\fhi\rc}}{\intg{\psi\rc}}
\xi(\gm,\rl,\rll,\rc,\psi,\fhi),\f{\intgc{\fhi}}{\intgc{\psi}}\eta(\gm,\gmt,\gmc,\rc,\psi,\fhi)\r\}
$$
Now, we prove that the $\Theta$-diameter  of $\L\l(C(b,c,\alpha)\r)$ is finite.
\begin{proposition} \label{DiametroFinito}
For all sufficiently large $b>0$, $c>0$ and for $\alpha \in (0,1]$ we have 
$$
\Delta:=\sup\l\{\Theta\l(\L\fhi,\L\psi\r); \fhi,\psi \in C(b,c,\alpha)\r\}<\infty.
$$
\end{proposition}
\begin{proof}
Given $\fhi,\psi \in C(\sigma b,\sigma c, \alpha)$, note that
$$\f{1-\sigma}{1+\sigma} < \xi(\gm,\rl,\rll,\rc,\psi,\fhi) < \f{1+\sigma}{1-\sigma}$$
and
$$\f{1-\sigma}{1+\sigma} <\eta(\gm,\gmt,\gmc,\rc,\psi,\fhi) < \f{1+\sigma}{1-\sigma}$$
Indeed, given $\rl,\rll,\rc \in\Dnorm{\gm}$ 
$$
\f{\intg{\fhi\rll}-\intg{\fhi\rl}}{\intg{\fhi\rc}} \leq \f{\sigma b \theta(\rl,\rll) \infimo{\rho \in \Dnorm{\gm}}{\intg{\fhi\rho}}}{\intg{\fhi\rc}} \leq \sigma b\theta(\rl,\rll)
$$
and
$$
\f{\intg{\fhi\rll}-\intg{\fhi\rl}}{\intg{\fhi\rc}} \geq \f{-\sigma b \theta(\rl,\rll) \infimo{\rho \in \Dnorm{\gm}}{\intg{\fhi\rho}}}{\intg{\fhi\rc}} \geq -\sigma b\theta(\rl,\rll)
$$
holds.

The same is valid for $\psi$ and as $\sigma < 1$ we conclude that
$$
\f{1-\sigma}{1+\sigma}
< 
\f{\l(\intg{\psi\rll}-\intg{\psi\rl}\r)/\intg{\psi\rc}+b\theta(\rl,\rll)}{\l(\intg{\fhi\rll}-\intg{\fhi\rl}\r)/\intg{\fhi\rc}+b\theta(\rl,\rll)}
<
\f{1+\sigma}{1-\sigma}
$$
That is,
$$\f{1-\sigma}{1+\sigma} < \xi(\gm,\rl,\rll,\rc,\psi,\fhi) < \f{1+\sigma}{1-\sigma}.$$
In a similar way, we prove that  
$$\f{1-\sigma}{1+\sigma} <\eta(\gm,\gmt,\gmc,\rc,\psi,\fhi) < \f{1+\sigma}{1-\sigma}$$
Denoting by $\Theta_+$ the projective metrics associated to the cone defined just by condition (A),  
$$
\Theta_+\l(\fhi,\psi\r)=\ds \log \sup_{\gm,\rho \in \D(\gm),\gmc,\rc\in\D(\gmc)}\l\{\f{\intg{\fhi\rho}\intgc{\psi\rc}}{\intgc{\fhi\rc}\intg{\psi\rho}}\r\}
$$
it followd by the expression of  $\Theta$ that
$$
\Theta(\fhi,\psi) < \Theta_+(\fhi,\psi)+\log\l(\f{1+\sigma}{1-\sigma}\r)^2.
$$
So, we just need to prove that the  $\Theta_+$-diameter of $\L\l(C(b,c,\alpha)\r)$ is finite.
By a triangular argument, it is sufficient to show that  $\{\Theta_+(\L\fhi,1); \fhi \in C(b,c,\alpha)\}$ is finite. 
For that, we just need to find an upper bound for
$
\f{\intgc{\L\fhi\rc}}{\intg{\L\fhi\rho}}
$
for all $\fhi \in C(b,c,\alpha)$, $\rho \in \Dnorm{\gm}$ and $\rc \in \Dnorm{\gmc}$.
First, note that 
$$
\f{\intgc{\L\fhi\rc}}{\intg{\L\fhi\rho}}=\f{\sumjap\intgcj{\fhi\rcj}}{\sumjap\intgj{\fhi\rj}}
$$
and we reduce our problem to bound
$$\f{\intgcj{\fhi\rcj}}{\intgj{\fhi\rj}}
=
\f{\intgcj{\fhi\rcj}}{\intgcj{\fhi}}
\f{\intgcj{\fhi}}{\intgj{\fhi}}
\f{\intgj{\fhi}}{\intgj{\fhi\rj}}$$
Denoting $\f{\rj}{\intgj{\rj}}$ and $\f{\rcj}{\intgcj{\rcj}}$ by $\rbj$ and $\rbbj$, respectively,
applying $(B)$ and lemma \ref{invconeaux}, we obtain
%
\begin{eqnarray}
\f{\intgcj{\fhi\rcj}}{\intgcj{\fhi}}
&=&
\f{\intgcj{\fhi\rbbj}\intgcj{\rcj}}{\intgcj{\fhi}} \leq\, 
\l(1+b\theta_j\l(\rbbj,1\r)\r)\intgcj{\rcj} \nonumber\\
&\leq&
\l(1+b\log\l(\f{1+\lambda}{1-\lambda}\r)\r)\intgcj{\rcj}\label{eqrep1}
\end{eqnarray}
and
$$
\f{\intgj{\fhi}}{\intgj{\fhi\rj}}=\f{\intgj{\fhi}}{\intgj{\fhi\rbj}\intgj{\rj}}
\leq
\f{\l(1+b\theta_j\l(1,\rbj\r)\r)}{\intgj{\rj}}
\leq
\f{\l(1+b\log\l(\f{1+\lambda}{1-\lambda}\r)\r)}{\intgj{\rj}} 
$$
We know that $\rcj = \f{1}{p}\rc\circ f \cdot e^{\phi}$ and $\rj = \f{1}{p}\rho\circ f \cdot e^{\phi}$. Since $\rho$ and $\rc$ are normalized, 
it follows that $\rcj \leq \f{1}{p}(1+\kappa diam(M)^\alpha)e^{\phi}$ and $\rj \geq \f{1}{p}(1+\kappa diam(M)^\alpha)^{-1}e^{\phi}$. 
Therefore
$$
\f{\intgcj{\rcj}}{\intgj{\rj}} < (1+\kappa diam(M)^\alpha)^2\f{\intgcj{e^{\phi}}}{\intgj{e^{\phi}}},
$$
On the other hand, $\holder{e^{\phi}} < \varepsilon \inf{e^{\phi}}$
 and so $\sup{e^{\phi}}<(1+\varepsilon diam (M)^{\alpha})\inf{e^{\phi}}$ . We then obtain
$
\f{\intgcj{e^{\phi}}}{\intgj{e^{\phi}}}<1+\varepsilon diam (M)^{\alpha}
$
and by consequence
$$
\begin{array}{rcl}
\f{\intgcj{\rcj}}{\intgj{\rj}} & < & (1+\kappa diam(M)^\alpha)^2(1+\varepsilon diam (M)^{\alpha}). \\
\end{array}
$$
Moreover for $\gm$ e $\gmc$ such that we can apply  (C) we have
$$
\f{\intgcj{\fhi}}{\intgj{\fhi}} \leq 1+cd\l(\gmcj,\gmj\r)^\alpha \leq 1+c.diam(M)^\alpha
$$
implying that
$$
\f{\intgcj{\fhi\rcj}}{\intgj{\fhi\rj}} < \l(1+b\log\l(\f{1+\lambda}{1-\lambda}\r)\r)^2(1+\max\{\kappa,c,\varepsilon\} diam(M)^\alpha)^4,
$$
finishing the proof of the proposition.
\end{proof}

\section{Exponential Decay of Correlations}

In this section, we prove the Exponential Decay of Correlation for  H\"older continuous observables. 

In our context, the transfer operator is just $\L(\fhi)=\fhi\circ f^{-1}$ acting in the space of continuous observables. 

The adjoint operador of $\L$ is
$$
\ds\int{\Lt\fhi \dm}=\ds\int{\fhi d\L^* \mu}.
$$
for all continuous $\fhi$ and all probability measure $\mu$.
Instead of the nonuniformly expanding case, any invariant probability  is an 
eigenmeasure of the transfer operator's adjoint:

\begin{proposition}
If  $f$ is invertible, then $\L^*(\mu)=\mu$ if and only if $\mu$ is $f$-invariant.
\end{proposition}

\begin{proof}

Let $\fhi$ be a continuous function.
If $\L^*(\mu)=\mu$ then
$$
\int{\fhi \circ f^{-1}}\dm=\int{\L(\fhi)}\dm=\int{\fhi} d\L^*(\mu)=\int{\fhi} \dm,
$$
Now, givem an $f$-invariant measure  $\mu$, we have
$$
\int{\fhi}d\L^*(\mu)=\int{\Lt(\fhi)}\dm=\int{\fhi\circ f^{-1}}\dm=\int{\fhi}\dm
$$

\end{proof}

Other important relation obtained from the  $f$-invariance of a measure  $\mu$ is that
\begin{equation}\label{dualidade}
\int{\l(\fhi \circ f^n\r)\psi}\dm=\int{\fhi\L^n(\psi)}\dm
\end{equation}
Indeed, as $\Lt(\fhi)=\fhi \circ f^{-1}$ we have
$$
\int{\l(\fhi \circ f\r)\psi}\dm = \int{\fhi \circ f \circ f^{-1} \psi \circ f^{-1}}\dm=\int{\fhi \L(\psi)}\dm
$$
and by induction, 
$$
\int{\l(\fhi \circ f^n\r)\psi}\dm=\int{\fhi\L^n(\psi)}\dm.
$$

The exponential decay of correlations of the maximizing entropy measure will be a consequence of 
the strict invariance of the Main Cone that we proved in the last section, and the following

\begin{lemma}\label{lesoma}
For all $\fhi \in C^\alpha\l(M\r)$ there exists $K(\fhi)>0$ such that $\fhi+K(\fhi) \in C(b,c,\alpha)$.
\end{lemma}

\begin{proof}
First we  prove that there exists $K_3=K_3(\fhi)>0$ such that $\fhi  + K_3$ satisfies the condition (C) in the definition of cone 
$C(b,c,\alpha)$.
The projections between stable leaves guarantees that
$$\intg{\fhi}=\intgt{\fhi\circ\pi}$$
Given $\fhi \in C^\alpha\l(M\r)$ we have
$$
\f{\fhi(x)-\fhi(\pi(x))}{d(\gm,\gmt)} \leq \f{\fhi(x)-\fhi(\pi(x))}{d(\pi(x),x)} \leq \holder{\fhi}
$$
So
$$
\sup_{\gm,\gmt}\l\{\f{\l|\intg{\fhi}-\intgt{\fhi}\r|}{d(\gm,\gmt)}\r\} \leq \holder{\fhi} < \infty
$$
On the other hand, for $K>0$, all we have  $\infimo{\gm}{\intg{\l(\fhi+K\r)}}=\infimo{\gm}{\intg{\fhi}}+K$. It is sufficient to choose 
$K_3=K_3(\fhi)>0$ such that
$$
c\infimo{\gm}{\intg{\l(\fhi+K_3\r)}}>\holder{\fhi}
$$
In order to see that there exists  $K_2=K_2(\fhi)$ such that $\fhi+K_2$ satisfies the condition (B), just note that
$$
\sup_{\rl,\rll \in \Dnorm{\gm}}\l\{\f{\l|\intg{\fhi\rl}-\intg{\fhi\rll}\r|}{\theta\l(\rl,\rll\r)}\r\} < \infty.
$$
Indeed, as $\rl,\rll \in \Dnorm{\gm}$ we have $\f{\rl}{\rll} \leq e^{\theta\l(\rl,\rll\r)}$ and so, for all  bounded $\fhi$ 
$$
\begin{array}{rl}
\modulo{\intg{\fhi\rl}-\intg{\fhi\rll}}
 = &
\modulo{\intg{\l(\f{\rl}{\rll}-1\r)\fhi\rll}} \leq \intg{\modulo{\f{\rl}{\rll}-1}\modulo{\fhi}\rll}\\
&\\
\leq &
\sup\modulo{\f{\rl}{\rll}-1}\sup\fhi\sup\rll = \modulo{\sup\f{\rl}{\rll}-1}\sup\fhi\sup\rll\\
&\\
\leq &
\modulo{e^{\theta\l(\rl,\rll\r)}-1}\sup\fhi \sup\rll\\
\end{array}
$$
Let $B$ such that $\sup{\l(\fhi+B\r)}=1$. It follows that
$$
\f{\modulo{\intg{\fhi\rl}-\intg{\fhi\rll}}}{\theta\l(\rl,\rll\r)}
= \f{\modulo{\intg{\l(\fhi+B\r)\rl}-\intg{\l(\fhi+B\r)\rll}}}{\theta\l(\rl,\rll\r)} 
$$
$$
\qquad \qquad \leq \f{\l(e^{\theta\l(\rl,\rll\r)}-1\r)\sup\rll}{\theta\l(\rl,\rll\r)}.
$$
If $\theta\l(\rl,\rll\r)<1$ then $\f{e^{\theta\l(\rl,\rll\r)}-1}{\theta\l(\rl,\rll\r)} <2$ and as $\rll \in \Dnorm{\gm}$ we have 
$$
\f{\modulo{\intg{\fhi\rl}-\intg{\fhi\rll}}}{\theta\l(\rl,\rll\r)}
\leq 2(1+\kappa diam(M)^\alpha)
$$
Now, if $\theta\l(\rl,\rll\r) \geq 1$ we obtain
$$
\begin{array}{rcl}
\f{\modulo{\intg{\fhi\rl}-\intg{\fhi\rll}}}{\theta\l(\rl,\rll\r)}
& \leq & \modulo{\intg{\l(\fhi+B\r)\rl}-\intg{\l(\fhi+B\r)\rll}}\\
& \leq & \intg{\modulo{\l(\fhi+B\r)\l(\rl-\rll\r)}}\\
&\leq & \sup\l(\fhi+B\r)\l(\sup\rl+\sup\rll\r) \leq \, 2(1+\kappa diam(M)^\alpha)\\
\end{array}
$$
and this implies
$$
\sup_{\rl,\rll \in \Dnorm{\gm}}\l\{\f{\l|\intg{\fhi\rl}-\intg{\fhi\rll}\r|}{\theta\l(\rl,\rll\r)}\r\} < \infty.
$$
The choice of $K_2=K_2(\fhi)$ is similar of what we have done for (C).
On condition (A) , since $\fhi$ is  continuous with compact domain, there exists  $K_1=K_1(\fhi)$ such that $\fhi+K_1>0$ and so $\intg{\l(\fhi+K_1\r)\rho}>0$, $\forall \gm \in F^s_{loc}$ and $\rho \in \D(\gm)$.
We complete the proof by taking  $K(\fhi)=\max\{K_1,K_2,K_3\}$.

\end{proof}

Now, denote by  $\mg \times \nu$ the measure given by 
$$
\mg \times \nu(\fhi):=\ds\int{\intg{\fhi}}\dn(\gm).
$$
By unicity of the maximal entropy probability measure, we notice that $\mu=\mg \times \nu$, where 
$\nu$ is the maximal entropy probability measure for $g$. 
Indeed,  let us first show that $\mg\times\nu$ is an $f$-invariant probability. In fact, for all $x \in M$, 
given $\gm=\Pi^{-1}_{\Lambda}(x)$ and $\gmj=\Pi^{-1}_{\Lambda}(x_j)$, with $f(\gmj)\subset \gm$ and $g(x_j)=x$ 
we have  $\mg(A)=\f{1}{p}\sumjap\mgj\l(f^{-1}(A)\r)$. 
By Castro-Varandas[CV13], $\nu$   is an  eigenmeasure of the adjoint $\L^*_{g,\phi}$ given by
$$
\L_{g,\phi}(\fhi)(x):=\ds\sum_{g(x_j)=x} e^{\phi(x_j)}\fhi(x_j), 
$$
for constant potential $\phi$.
More precisely, if $r$ is the spectral radius of  $\L^*_{g,\phi}$, which is equal to the degree of $g$,
then  $\L^*_{g,\phi}(\nu)=r\nu$. By normalizing $\L^*_{g,\phi}$ by $r= p$, we obtain for any continuous $\phi$
$$
\ds\int{\fhi(x)}d\nu=\f{1}{r}\int{\fhi(x)}d\L^*_{g,\phi}\nu=\ds\f{1}{r}\int{\L_{g,\phi}(\fhi)(x)}d\nu=\int\f{1}{p}\sumjap{\fhi(x_j)}d\nu.
$$
Therefore, for $A \in \A_0$ we deduce
$$
\begin{array}{rcl}
(\mg \times \nu)(f^{-1}(A))&=&\mg \times \nu(\chi_{f^{-1}(A)})= \ds\int{\intg{\chi_{f^{-1}(A)}}}d\nu\\
&=&\ds\int{\mg\l(f^{-1}(A)\r)}d\nu=\int\f{1}{p}\sumjap{\mgj\l(f^{-1}(A)\r)}d\nu\\
&=&\ds\int{\mg(A)}d\nu=\ds\int{\intg{\chi_{A}}}d\nu = \, \mg \times \nu(A)\\
\end{array}
$$
As we have shown in previous sections, this implies the same equality for any borelian $A$.

Furthermore $\mg\times\nu(A)=\mu(A)$ . Indeed, let $A=\Pi_{\Lambda}^{-1}(A_N)$, with $A_N\in\A_N$.

On the one hand, we have that
$$
\mu(\Pi_{\Lambda}^{-1}(A_N))=\nu\l(\Pi_{\Lambda}\l(\Pi_{\Lambda}^{-1}(A_N)\r)\r)=\nu\l(A_N\r)=\ds\int_{N}\chi_{A_N}d\nu
$$
and on the other hand, 
$$
\mu_\gm\times\nu\l(\Pi_{\Lambda}^{-1}(A_N)\r)=\ds\int\intg{\chi_{\Pi_{\Lambda}^{-1}(A_N)}}d\nu
$$

As $\ds\chi_{\Pi_{\Lambda}^{-1}(A_N)}(x)=\chi_{A_N}(\Pi_{\Lambda}(x))$ and for all $\gm$ 
there exists $x_0 \in N$ 
such that $\gm=\Pi^{-1}_{\Lambda}\l(x_0\r)$.
So
$$
\intg{\chi_{\Pi_{\Lambda}^{-1}(A_N)}(x)}=\intg{\chi_{A_N}(\Pi_{\Lambda}(x))}=\intg{\chi_{A_N}(x_0)}=\chi_{A_N}(x_0)
$$

and then,  $\mg\times\nu(A)=\mu(A)$ for all $A\in\A_0$. Now, given $A \in \A= \ds\bigcup_{n=0}^{\infty} \A_n$, as $\A_n=f^n(\A_0)$, we have that there exist $n \in \N$ and $A_0 \in \A_0$ such that $A=f^n(A_0)$. 
Therefore
$$
\mg\times\nu(A)=\mg\times\nu(f^n(A_0))=\mg\times\nu(A_0).
$$
Since $\mu$ is $f$-invariant, $\mu(A)=\mu(f^n(A_0))=\mu(A_0)$, we conclude that $\mu=\mg \times \nu$.

\begin{teoremaB} \label{teoB}
The measure $\mu$ has exponential decay of de correlations for  H\"older continuous observables.
\end{teoremaB}

\begin{proof}

We should prove that for $\alpha$-H\"older observables $\fhi,\psi$ , there exist $0<\tau<1$ and $K(\varphi,\psi)>0$ such that
\begin{equation*}
\l|\int (\varphi\circ f^n) \psi \dm - \int \varphi \dm\int \psi \dm \r|
	\leq K(\varphi,\psi)\cdot\tau^n, \forall n \geq 1.
\end{equation*}
By (\ref{dualidade}) this is equivalent to prove
\begin{equation*}
\l|\int \varphi \Lt^n\l(\psi\r) \dm - \int \varphi \dm \int \psi \dm\r|
	\leq K(\varphi,\psi)\cdot\tau^n,
	\quad \text{for all $n\ge 1$}.
\end{equation*}
We start with the case $\fhi_{|\gm} \in \D\l(\gm\r)$,  $\forall \gm \in \F^s_{loc}$ and $\psi \in C(b,c,\alpha)$. We also assume $\int{\fhi}\dm \neq 0$ and $\int{\psi}\dm=1$.

Recall that $\L(1)=1\circ f=1$. Since $\fhi_{|\gm} \in \D(\gm)$ for all $\gm \in\F^s_{loc}$  by (A) we have
$$\f{\intg{\fhi\Lt^n(\psi)}}{\intg{\fhi}} \leq \beta_+\l(\Lt^n(\psi),1\r)$$
Since $\psi$ is normalized  we have $\ds\int{\Lt^n(\psi)}\dm=\ds\int{\psi}\dm=1$.
As $\mu=\mg\times\nu$
$$
\ds\int{\l(\intg{\Lt^n(\psi)}\r)}d\nu=\ds\int{\Lt^n(\psi)}\dm=1
$$
and so there exists $\gmc$ such that
$\intgc{\Lt^n(\psi)}\leq 1$. We conclude that
$$
\alpha_+\l(\Lt^n(\psi),1\r) \leq \f{\intgc{\Lt^n(\psi)}}{\intgc{}} =\intgc{\Lt^n(\psi)} \leq 1
$$
and for all $\gm \in\F^s_{loc}$
$$
\f{\intg{\fhi\Lt^n(\psi)}}{\intg{\fhi}} \leq \f{\beta_+\l(\Lt^n(\psi),1\r)}{\alpha_+\l(\Lt^n(\psi),1\r)} \leq e^{\Theta_+\l(\Lt^n(\psi),1\r)} \leq  e^{\Theta\l(\Lt^n(\psi),1\r)}.
$$
By proposition $\ref{InvarianciaDoCone}$ and by proposition $\ref{DiametroFinito}$,  since the cone $C\l(\sigma b,\sigma c,\alpha\r)$ has $\Theta$-diameter less or equal than $\Delta$, it follows from proposition $\ref{t.Birkhoff}$ that  $\exists \, 0<\tau<1$ such that $\forall \fhi,\psi \in C\l(b,c,\alpha\r)$ we have $\Theta(\Lt^n(\fhi),\Lt^n(\psi)) \leq \Delta \tau^{n-1}$. In consequence,
$$
\f{\ds\int{\fhi\Lt^n(\psi)}\dm}{\ds\int{\fhi}\dm}=\f{\ds\int{\intg{\fhi\Lt^n(\psi)}}d\nu}{\ds\int{\intg{\fhi}}d\nu}\leq e^{\Theta\l(\Lt^n(\psi),1\r)} \leq e^{\Delta \tau^{n-1}}.
$$
Note now that $\ds\lim_{n \to \infty}\f{e^{\Delta\tau^{n-1}}-1}{\tau^n}=\f{\Delta}{\tau}$. So, there exists $\tilde{\Delta}>0$ 
such that $e^{\Delta\tau^{n-1}}-1 < \tilde{\Delta}\tau^n$, for all $n \in \N$. This implies that
$$
\l|\ds\int{\fhi}\dm\r|\l|\f{\ds\int{\fhi\Lt^n(\psi)}\dm}{\ds\int{\fhi}\dm}-1\r| \leq \modulo{\ds\int{\fhi}\dm}\l(e^{\Delta \tau^{n-1}}-1\r) \leq \modulo{\ds\int{\fhi}\dm}\tilde{\Delta}\tau^n
$$
If $\int{\psi}\dm \neq 1$ then

\begin{equation*}
\begin{array}{rcl}
	\l|\ds\int\varphi \Lt^n\l(\psi\r) \dm - \ds\int\varphi \dm \ds\int\psi \dm\r|
	&=&
	\l|\ds\int{\psi}\dm\r|\l|\ds\int\varphi \Lt^n\l(\f{\psi}{\int{\psi}\dm}\r) \dm - \ds\int\varphi \dm \r|\\
	&\leq&
	\l|\ds\int{\psi}\dm\r|\modulo{\ds\int{\fhi}\dm}\tilde{\Delta}\tau^n\\
\end{array}	
\end{equation*}
for all $n \geq 1$.

 By lemma \ref{lesoma} given  an $\alpha$-H\"older continuous function $\psi$, there exists  $K(\psi)>0$, such that $\psi+K(\psi) \in C(b,c,\alpha)$. 
 Therefore
$\psi=\psi+K(\psi)-K(\psi)$ and noting that $\int{\fhi\L^n(K(\psi))}\dm = \int{\fhi}\dm\int{K(\psi)}\dm$ we obtain
\begin{equation*}
\begin{array}{rl}
	\l|\ds\int\varphi \Lt^n\l(\psi\r) \dm - \ds\int\varphi \dm \ds\int\psi \dm\r|

	=&
	\l|\ds\int\varphi \Lt^n\l(\psi+K(\psi)\r) \dm -\int{\fhi}\dm\int{\l(\psi+K(\psi)\r)}\dm\r|\\
	\leq&
	\l(\modulo{\ds\int{\psi}\dm}+K(\psi)\r)\modulo{\ds\int{\fhi}\dm}\tilde{\Delta}\tau^n\\
\end{array}	
\end{equation*}
Now, given an  $\alpha$-H\"older $\fhi$, note that  there exists $K(\fhi)  \in \R$ such that
 $\fhi_{|\gm}+K(\fhi)+B \in \D\l(\gm\r)$ for all  $\gm \in \F^s_{loc}$ and
  $\ds\int{\fhi+K(\fhi)}+B\dm>0$, for all $B>0$ . Indeed, 
$$
\holder{\fhi_{|\gm}+K(\fhi)} < \kappa \inf\l\{\fhi_{|\gm}+K(\fhi)\r\}
$$
if, and only if,
$$
K(\fhi) > \f{\holder{\fhi_{|\gm}}}{\kappa}-\inf\l\{\fhi_{|\gm}\r\}
$$
Set $K(\fhi) = \ds\sup_{\gm \in \F^s_{loc}}\l\{\modulo{\f{\holder{\fhi_{|\gm}}}{\kappa}}\r\}-\inf{\fhi}$. 
Observe that $K(\fhi)\leq\f{\holder{\fhi}}{\kappa}-\inf{\fhi}<\infty$.  
As $\fhi_{|\gm}+K(\fhi)\geq \f{\holder{\fhi_{|\gm}}}{\kappa} \geq 0$ for all $\gm \in \F^s_{loc}$, it follows that
 $\fhi_{|\gm}+K(\fhi)+B \in \D\l(\gm\r)$ and $\ds\int{(\fhi+K(\fhi))\dm}+B>0$,  $\forall B>0$. Analogously to the last case
\begin{equation*}
\begin{array}{rl}
	\l|\ds\int\varphi \Lt^n\l(\psi\r) \dm - \ds\int\varphi \dm \ds\int\psi \dm\r|
	\leq&
	\l(\modulo{\ds\int{\psi}\dm}+K(\psi)\r)\l(\modulo{\ds\int{\fhi}\dm}+K(\fhi)+B\r)\tilde{\Delta}\tau^n\\
\end{array}	
\end{equation*}
 and since $B$ is any positive number
\begin{equation*}
\begin{array}{rl}
	\l|\ds\int\varphi \Lt^n\l(\psi\r) \dm - \ds\int\varphi \dm \ds\int\psi \dm\r|
	\leq&
	\l(\modulo{\ds\int{\psi}\dm}+K(\psi)\r)\l(\modulo{\ds\int{\fhi}\dm}+K(\fhi)\r)\tilde{\Delta}\tau^n\\.
\end{array}	
\end{equation*}
Since $\modulo{\ds\int{\fhi\dm}}-\inf{\fhi}\geq 0$, we have $\modulo{\ds\int{\fhi}\dm}+K(\fhi) \geq 0$. By taking 
$$
K(\fhi,\psi):=\l(\modulo{\ds\int{\psi}\dm}+K(\psi)\r)\l(\modulo{\ds\int{\fhi}\dm}+K(\fhi)\r)\tilde{\Delta},
$$
we conclude the proof of the Theorem.

\end{proof}

\section{Central Limit Theorem}

Let $\G$ be the Borel $\sigma$-algebra of $M$ and let $\G_n := f^{-n}(\G)$ be a nonincreasing family of  $\sigma$-algebras. 
A function $\xi : M \to \R$ is $\G_n$-measurable if, and only if, $\xi=\xi_n \circ f^n$ for some  $\G$- measurable $\xi_n$. 
Let $L^2(\G_n)= \{\xi \in L^2\l(\mu\r); \xi$ is  $\G_n$-measurable $\}$. Note that $\L^2(\G_{n+1}) \subset \L^2(\G_n)$ for each $n\geq 0$. Given
$\fhi \in L^2(\mu)$, we will denote by $\E(\fhi|\G_n)$ the $L^2$-orthogonal projection of $\fhi$ on $L^2(\G_n)$.

We will apply   the following adaption of a result due to Gordin,  whose proof can be found in [Vi97]:

\begin{theorem} \label{TeoremaDeGordin}[Gordin.]
Let $\l(M,\F,\mu\r)$ be a probability space, and let $\phi \in L^2(\mu)$ be such that $\int{\phi}\dm=0$. 
Assume that  $f:M \seta M$ is an invertible bimeasurable map and that $\mu$ is an $f$-ergodic invariant probability. 
Let $\F_0 \subset \F$ such that $\F_n:=f^{-n}(\F_0)$, $n \in \Z$, is a nonincreasing family of $\sigma$-algebras. Define 
$$
\sigma_\phi^2:=\int \phi^2 d\mu + 2\sum\limits_{j=1}^{\infty} \phi\cdot (\phi\circ f^j) \, d\mu.
$$
Assume
$$
\ds\sum_{n=0}^\infty \norma{\E(\phi|\F_n)}_2 < \infty \text{ and } \ds\sum_{n=0}^\infty \norma{\phi-\E(\phi|\F_{-n})}_2 < \infty.  
$$
Then $\sigma_\phi<\infty$ and $\sigma_\phi=0$ if, and only if, $\phi=u\circ f - u$ for some $u \in L^1(\mu)$. Moreover, if $\sigma_\phi>0$ then
 for any interval  $A \subset \R$

\begin{equation*}
\mu\left(x\in M: \frac{1}{\sqrt{n}}\sum\limits_{j=0}^{n-1}
\left(\phi(f^j(x))\right)\in A\right)\to \frac{1}{\sigma_\phi\sqrt{2\pi}}
\int_A e^{-\frac{t^2}{2\sigma_\phi^2}} dt,
\end{equation*}
as $n\to\infty$.
\end{theorem}

Let $\F_0$ the $\sigma$-algebra whose elements are Borelian subsets  of $\Lambda$ which are union local stable leaves 
(intersected with $\Lambda$).
Not that, if $\fhi$ é $F_0$-mensurable then $\fhi$ is constant along local stable leaves.

 We start by proving a statement of exponential decay of correlation concerning to function in $L^1\l(F_0\r)$.
 
 \begin{proposition}
Let $\fhi \in L^1\l(F_0\r)$ and $\psi$ be a $\alpha$-H\"older continuous function. 
Then, there exist constants $0<\tau<1$ and $C(\psi)>0$ such that
$$
	\ds\l|\int (\fhi\circ f^n) \psi \dm - \int \fhi \dm\int \psi \dm \r|	\leq C(\psi)\int{\modulo{\fhi}\dm}\cdot\tau^n
$$
for all $n \geq 1$.
\end{proposition}

\begin{proof}

Since $\fhi$ is $F_0$-measurable,  it is constant restricted to local stable leaves, so, $\holder{\fhi|_{\gm}}=0$, $\forall \gm \in \F^s_{loc}$. 
Suppose $\fhi \geq 0 $ and let $K(\fhi)$ and $K(\psi)$ as in the proof of Th. \ref{thB}. 
Therefore   
$$
	K(\fhi) = \ds\sup_{\gm \in \F^s_{loc}}\l\{\modulo{\f{\holder{\fhi_{|\gm}}}{\kappa}}\r\}-\inf{\fhi} = -\inf{\fhi}
$$
Since $\ds\modulo{\int{\fhi\dm}}-\inf{\fhi} \leq \int{\modulo{\fhi}\dm}$, just as in the proof of Th. \ref{thB}, it follows that
$$
	\ds\l|\int (\fhi\circ f^n) \psi \dm - \int \fhi \dm\int \psi \dm \r|	\leq \l(\modulo{\ds\int{\psi}\dm}+K(\psi)\r)\int{\modulo{\fhi}\dm}\cdot\tau^n.
$$
Now, we can write  $\fhi=\fhi^+-\fhi^-$ where $\fhi^{\pm}=\f{1}{2}\l(\modulo{\fhi}\pm \fhi\r)$. Noting that
 $\ds\int{\modulo{\fhi^{\pm}}\dm} \leq \int{\modulo{\fhi}\dm}$ from linearity of  the integral we obtain
$$
	\ds\l|\int (\fhi\circ f^n) \psi \dm - \int \fhi \dm\int \psi \dm \r|	\leq C(\psi)\int{\modulo{\fhi}\dm}\cdot\tau^n
$$
with $C(\psi):=2\l(\modulo{\ds\int{\psi}\dm}+K(\psi)\r)$.
\end{proof}

As a consequence of the proposition we are able to prove:
\begin{lemma}
For every H\"older continuous function $\fhi$ with $\ds\int{\fhi}\dm=0$  there is $R=R(\fhi)$ 
such that $\norma{\E(\fhi|\F_n)}_2 \leq R\tau^n$
for all $n\geq 0$.
\end{lemma}

\begin{proof}

Due to the last proposition, if $\psi \in L^1(F_0)$ and $\ds\int{\psi}\dm \leq 1$, then
$$
\ds\l|\int (\psi\circ f^n) \fhi \dm - \int \psi \dm\int \fhi \dm \r|	\leq C(\fhi)\cdot\tau^n.
$$

As $\norma{\psi}_1 \leq \norma{\psi}_2$ and $\ds\int{\fhi}\dm=0$ we have

$$
\begin{array}{rcl}
\norma{\E(\fhi|\F_n)}_2
&=&
\sup\l\{\ds\int\xi\fhi\dm; \xi \in L^2(\F_n), \norma{\xi}_2=1\r\}\\
&=&\sup\l\{\ds\int\l(\psi \circ f^n\r)\fhi\dm; \psi \in L^2(F_0), \norma{\psi}_2=1\r\} \leq
R\l(\fhi\r)\tau^n
\end{array}
$$
\end{proof}

Now, we can prove:

\begin{teoremaC}{\textbf{(Central Limit Theorem)}}

Let $\mu$ be the maximal entropy probability for $f:\Lambda \to \Lambda$, as in (\ref{contextoprincipal}). 
Given a H\"older continuous function $\varphi$  and  
$$
\sigma_\varphi^2:=\int \phi^2 d\mu + 2\sum\limits_{j=1}^{\infty}\int \phi\cdot (\phi\circ f^j) \, d\mu,
\quad \text{ with } \quad \phi=\varphi-\int \varphi \, d\mu.
$$
Then $\sigma_\varphi<\infty$ and $\sigma_\varphi=0$ if, and only if,  $\varphi=u\circ f - u$ for some $u \in L^1(\mu)$. 
Moreover, if $\sigma_\varphi>0$ then for all interval $A\subset\real$
\begin{equation*}
	\lim_{n\to\infty}\mu\left(x\in M: \frac{1}{\sqrt{n}}\sum\limits_{j=0}^{n-1}
	\left(\varphi(f^j(x))-\int \varphi d\mu\right)\in A\right)= \frac{1}{\sigma_\varphi\sqrt{2\pi}}
	\int_A e^{-\frac{t^2}{2\sigma_\varphi^2}} dt.
\end{equation*}
\end{teoremaC}
\begin{proof}
By the last lemma,  $\ds\sum_{n=0}^\infty \norma{\E(\phi|\F_n)}_2 < \infty$, so the first condition for Gordin's Theorem holds. The second condition  
follows from the  H\"older continuity of $\fhi$. In fact,
$\E(\phi,\F_{-n})$ is constant in each  $n$-image $\eta=f^n(\gm)$ of a stable leaf $\gm$ and
$$
\inf(\phi|_{\gm}) \leq \E(\phi,\F_{-n}) \leq \sup(\phi|_{\gm}).
$$
Since the diameter of  $\eta$  is less $C_s\lambda_s^n$ for some constant $C_s$ which does not depend on  $\gm$, $\lambda_s \in (0,1)$, and $\phi$ is $(A,\alpha)$-H\"older for some constant $A>0$, we obtain that 
$$
\norma{\phi - \E(\phi,\F_{-n}) }_2 \leq \norma{\phi - \E(\phi,\F_{-n}) }_0 \leq AC_s^{\alpha}\lambda_s^{\alpha n}.
$$
which guarantees $\ds\sum_{n=0}^\infty \norma{\phi - \E(\phi,\F_{-n}) }_2 < \infty$. The result then follows
as a consequence of Gordin's Theorem.
\end{proof}

\begin{remark}
We have seen that the invariance a suitable cone of distributions with respect the transfer operator implied in a rather economic way the exponential decay of correlations and the Central Limit. Once we have such cone, one
can define a (basically unique) anisotropic space $E$ associated to such cone. Such space is a Banach space
in which the transfer operator exhibits  a spectral gap.  The procedure to define such space is very simple.
Roughly speaking, such cone $C$ defines a order relation $\prec$  in a Banach Lattice $B$ of bounded functions given by
$$   
v \prec w  \Leftrightarrow w- v \in C
$$
Let $e$ be some function in $B$ such that for any function $\varphi$ in $B$ , there exists a constant $c_\varphi$ such that  
$
-c_\varphi  e \prec \varphi \prec c_\varphi e.
$
For example, in our case, take $e \equiv 1$.
Setting $\|\varphi\| := \inf \{c_\phi; -c_\varphi  e \prec \varphi \prec c_\varphi e \}$, one easily checks 
that this is a norm. The anisotropic space is obtained by completing $B$ with respect this norm. 
This approach, associated with Lasota Yorke estimates instead of projective cones,  was used by Baladi, Gouezel, Liverani, Tsujii,  among others, in several works \cite{ Bal05, BT07, BG10, BL12} to study fine spectral properties of the transfer operator  in hyperbolic contexts. 
\end{remark}

\section{Systems derived from Anosov}
In this section we study  partially hiperbolic dynamics as in the second setting described in section \ref{contexto}. Recall that such setting contains an open class of
 derived from Anosov systems. 
Using the techniques and results in the sections before, \label{secanosov} 
we prove again the strict invariance of a suitable cone of functions by the transfer  operator. Furthermore, from the convergence of such cone,  
we construct a measure exhibiting  exponential decay of  correlations, and satisfying the Central Limit Theorem.

\subsection{Invariant Cones}
Given $x \in \Lambda$ we will denote by $\gm$ the intersection of the local (strong) stable manifold  $x$ with the Markov rectangle  where $x$ belongs in. By the mixing property of the Markov partition, there is no great loss of generality in supposing that for any
rectangles  $R_i$ and $R_j$ of $\cR$ we have 
$$ f(R_i)\cap R_j \neq \emptyset.$$
Otherwise, one can work with some positive iterate of $f$. 
Given $\gm= W^s_{loc}(x) \cap R_i$, for some fixed Markov rectangle $R_i$, by the Markovian property we have 
 $\gm = \ds\bigcup_{j=1}^{p_\gm} f(\gm_j)$, for $\gm_j, j= 1, \dots p_\gm$ corresponding to strong stable manifolds whose
 images by $f$ intersect the interior of $\gm$. 
Observe also that by the Markovian property, all $\gm$ in $R_i$ have pre-images in the same rectangles. In particular,
$p_\gm$ depends only on $R_i$.
This is  important for the calculation of the finite diameter of the cone and slightly modifies the definition  of the auxiliar probabilities $\mu_\ga$  
on the leaves $\gm \in \F^s_{loc}$.

Let $\gm \in \F^s_{loc}$,  $\gm = \ds\bigcup_{j=1}^{p_{\gm}} f(\gm_j)$. Given $n \in \N$, $n \geq 1$, by induction we can write 
$$
\gm = \ds\bigcup_{i_1=1}^{p_{i_0}}\bigcup_{i_2=1}^{p_{i_1}} \cdots \bigcup_{i_n=1}^{p_{i_{n-1}}} f^n(\gm_{i_1 \cdots i_n})
$$
\begin{figure}[!htb]
   \centering
   \includegraphics[scale=0.4]{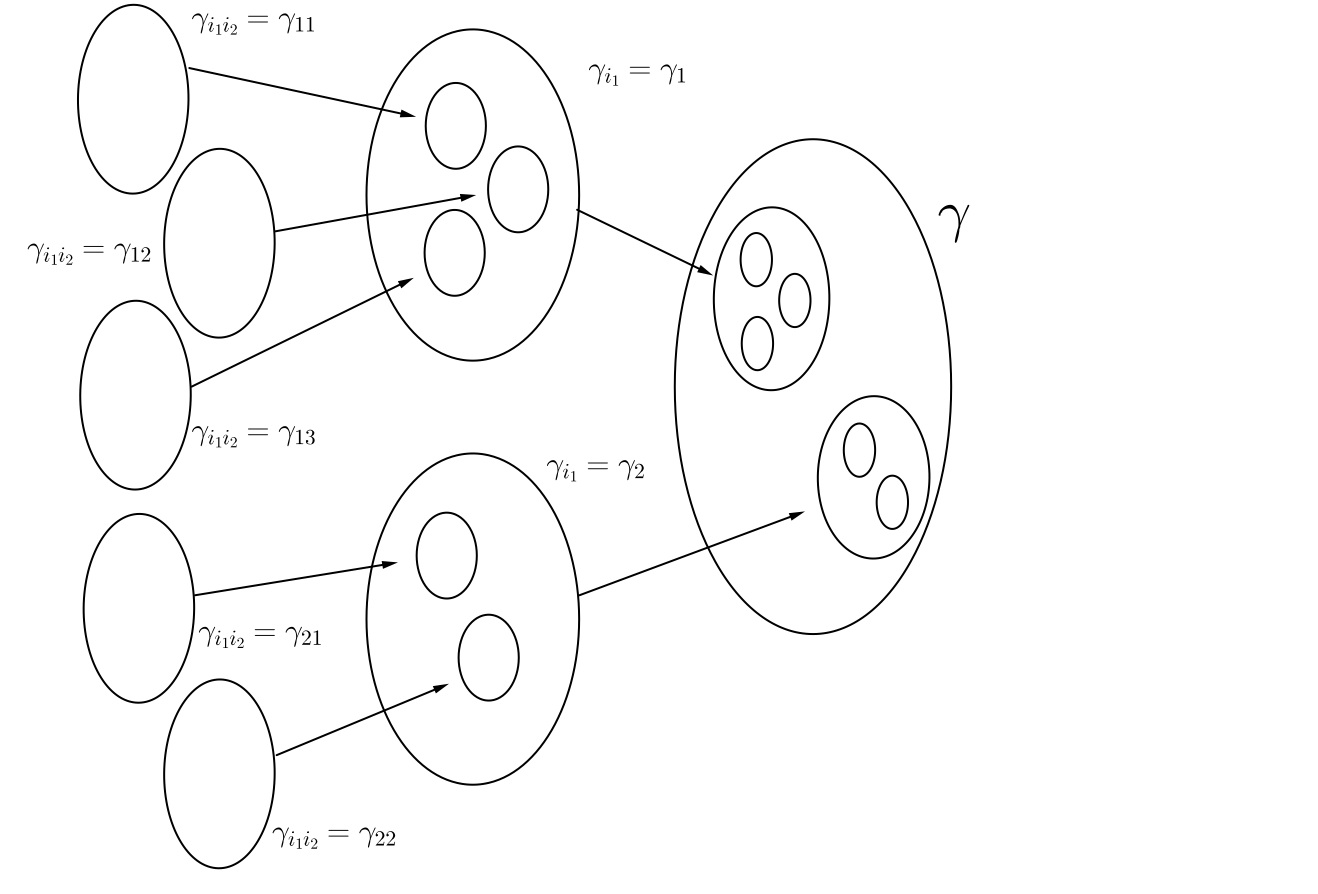}
   \caption{Mass distribution for derived from Anosov systems}
   \label{DistribuicaoDeMassaParcialmenteHiperbolico}
\end{figure}
where $f(\gm_{i_1\cdots i_{n+1}}) \subset \gm_{i_1 \cdots i_{n}}$, 
$\gm_{i_1 \cdots i_{n}} = \ds\bigcup_{i_{n+1}=1}^{p_{i_n}} f(\gm_{i_1 \cdots i_{n+1}})$ and $p_{i_k} , k \in \N$. 
Note also that que $p_{i_0}=p_{\gm}$. Since the last union is disjoint, 
for each $n \in \N$, $n \geq 1$, the sets $f^n(\gm_{i_1 \cdots i_n})$ with $i_k \in \l\{1, \cdots, p_{i_{k-1}}\r\}$, $k \in \N$, $k \geq 1$ define
a partition for $\gm$. So, it is natural to define a  probability measure in $\gm$ as:
$$
\mg\l(f^n\l(\gm_{i_1 \cdots i_n}\r)\r):=\f{1}{p_{i_0}p_{i_1}\cdots p_{i_{n-1}}}
$$
Since $f^n\l(\gm_{i_1 \cdots i_n}\r)$ we have that  $\mg$ is a Borelian measure of $\gm$. 
In fact, It is a  probability measure, as 
$$
\mg(\gm)=\ds\sum_{i_1=1}^{p_{i_0}}\sum_{i_2=1}^{p_{i_1}} \cdots \sum_{i_n=1}^{p_{i_{n-1}}} \mg \l(f^n(\gm_{i_1 \cdots i_n})\r)=
\ds\sum_{i_1=1}^{p_{i_0}}\sum_{i_2=1}^{p_{i_1}} \cdots \sum_{i_n=1}^{p_{i_{n-1}}} \f{1}{p_{i_0}p_{i_1}\cdots p_{i_{n-1}}} =1.
$$
Of course,  such measure can be seen as  a conditional measure on $\Lambda$. 
Let  $A$ be a Borelian subset of $\Lambda$. So,
$$
\mg(A)=\mg(A \cap \gm)=\mg\l(A \cap \ds\bigcup_{j=1}^{p_{\gm}}f(\gm_{j})\r)=\mg\l(\ds\bigcup_{j=1}^{p_{\gm}}\l(A \cap f(\gmj)\r)\r)=\ds\sum_{j=1}^{p_{\gm}}\mg(A \cap f(\gmj))
$$
Defining $\mgj(A):=p_{\gm} \mg(f(A \cap \gm_j))$ we obtain that $\f{1}{p_{\gm}}\mgj(f^{-1}(A))=\mg(A \cap f(\gmj))$ and then
$$
\mg(A)=\f{1}{p_{\gm}}\ds\sum_{j=1}^{p_{\gm}}\mgj(f^{-1}(A))
$$
Moreover,  $\mgj\l(f^n\l(\gm_{j_1 \cdots j_n}\r)\r)=\f{1}{p_{j_0}p_{j_1}\cdots p_{j_{n-1}}}$. In  particular, it is  a probability in $\gmj$.
Note also that $\mg(A \cap f(\gmj))=\f{1}{p_{\gm}}\mgj(f^{-1}(A))$ implies that for any measurable set $A$, 
its characteristic function $\chi_A$ satisfies
$$
\ds\int_{f(\gmj)}\chi_A\dmg=\f{1}{p_{\gm}}\intgj{\chi_A \circ f}
$$
therefore, by the dominated convergence Theorem, for any continuous function  $g:\Lambda \to \R$ one concludes that 
\begin{equation} \label{MudancaDeVariavelAnosov}
\ds\int_{f(\gmj)}g\dmg=\f{1}{p_{\gm}}\intgj{g \circ f}.
\end{equation}

Applying the transfer operator $\L(\fhi)(x)=\fhi(f^{-1}(x))e^{\phi(f^{-1}(x))}$ and using \ref{MudancaDeVariavelAnosov} we obtain
$$
\intg{\L(\fhi)\rho} =\sumjapg \intgj{\fhi\rj}
$$
with $\rj:=\f{1}{p_{\gm}}\rho\circ f e^{\phi}$.
We will adapt the same kind of Main Cone previously defined for a diffeomorphism semiconjugated to a Castro-Varandas map.

The main change is in condition (C). The comparison in condition (C) 
will concern just strong stable leaves in the same rectangle. 
In the same manner as before we define the cone 
$\D(\gm,\kappa)$ of densities $\rho:\gm \seta \R$ such that $\rho > 0$ and
$|\rho|_{\alpha}< \kappa \inf{\rho}$. Lemma  \ref{invconeaux}, is still valid, that is, there exists $0<\lambda<1$ and $\kappa>0$ such that
\begin{enumerate}
\item For all $\gm \in \F_{loc}^s$ if $\rho,\rc \in \D(\gm,\lambda\kappa)$ then $\theta(\rho,\rc)\leq 2\log\l(\f{1+\lambda}{1-\lambda}\r)$. 
\item If $\rho \in \D(\gm,\kappa)$ then $\rho_j \in \D(\gm_j,\lambda\kappa)$,  $\forall j \in\{1, \ldots,p\}$. 
\item If $\rl,\rll \in \D(\gm,\kappa)$ then there exists  $\Lambda_1 =1-\l(\f{1-\lambda}{1+\lambda}\r)^2$ such that $\theta_j(\rl_j,\rll_j)\leq \Lambda_1\theta(\rl,\rll)$, $\forall j \in\{1, \ldots,p\}$;
\end{enumerate}
where $\theta_j$ and $\theta$ are, respectively, the projective metrics of $\D(\gmj,\kappa)$ and $\D(\gm,\kappa)$.

Once more, denote by $\Dnorm{\gamma}$ the set of densities $\rho \in \D(\gm,\kappa)$  such that $\intg{\rho}=1$. Given $b>0$, $c>0$ and $\kappa$ as in lemma $\ref{invconeaux}$, let $C[b,c,\alpha]$ be the cone of functions $\fhi \in E$ satisfying the conditions below for any $\gm \in \F^s$:

\begin{itemize}
\item{(A)} For all $\rho \in \D(\gm,\kappa)$:
$$ \intg{\fhi\rho} >0$$
\item {(B)} For all  $\rl, \rll \in \Dnorm{\gamma}$:
$$
\modulo{\intg{\fhi\rl}- \intg{\fhi \rll}} < b \theta\l(\rl,\rll\r) \infimo{\rho \in \Dnorm{\gamma}}{\intg{\fhi\rho}}
$$
\item {(C)} Given to leaves $\gm$ and $\gmt$ in the same Markov rectangle $R_i$:
$$
\modulo{\intg{\fhi} - \intgt{\fhi}} < c d(\gm,\gmt)^{\alpha}\infimo{\gm}{\intg{\fhi}} 
$$
\end{itemize}

\begin{remark} \label{folhas_no_mesmo_retangulo}
Given $\gm$ and $\gmt$ in the same rectangle, that is,  $\gm=W^s_{loc}(x) \cap R_i$, $\gmt=W^s_{loc}(y) \cap R_i$ and $x,y \in int(R_i)$ then $W^s_{loc}(x) \cap f(R_j) \neq \emptyset$ if and only if $W^s_{loc}(y) \cap f(R_j) \neq \emptyset$. Therefore, there exists $p \in \N$ such that $\gm = \ds\bigcup_{j=1}^p f(\gm_j)$ and $\gmt = \ds\bigcup_{j=1}^p f(\gmc_j)$, where $\gm_j, \gmt_j \subset R_j$ and $p_\gm=p_{\gmt}=p$. 
\end{remark}

\begin{remark}\label{Distancia_Entre_Folhas} 
Let $\gm$ and $\gmt$ in the same Markov rectangle. Assume without loss of generality that $\gm_1$ and $\gmt_1$ are
in the same good rectangle and so there exists $0<\lambda_{uc}<1$ such that $d(\gm_1,\gmt_1) \leq \lambda_{uc}d(\gm,\gmt)$ 
and in the other cases, $j \neq 1$, there exists $\tilde{L}\geq 1$ close to 1 such that $d(\gm_j,\gmt_j) \leq \tilde{L}d(\gm,\gmt)$. 
\end{remark}

The proof of the next proposition is entirely analogous to Prop. \ref{InvarianciaDoCone}:
\begin{proposition} \label{InvarianciaDoConeAnosov}
There exists $0<\sigma<1$ such that $\L(C[b,c,\alpha]) \subset C[\sigma b,\sigma c,\alpha]$ for 
sufficiently big $b$ and $c$.
\end{proposition}

Let $\Theta$ denote the projective metrics of $C\l[b,c,\alpha\r]$. 
Now we occupy ourselves with the important
\begin{proposition} \label{DiametroFinitoAnosov}
For all sufficiently big $b>0$, $c>0$ and $\alpha \in (0,1]$ the  diameter of $\L(C[b,c,\alpha])$ 
is finite, that is, there exists $\Delta:=\sup\l\{\Theta\l(\L\fhi,\L\psi\r); \fhi,\psi \in C[b,c,\alpha]\r\}<\infty$.
\end{proposition}
\begin{proof}

Denote by  $\Theta_+$ the projective metrics associated to  the cone defined by condition  (A).

Exactly as in Proposition \ref{DiametroFinito}, 
it follows from the expression of $\Theta$ that
$$
\Theta(\fhi,\psi) \leq \Theta_+(\fhi,\psi)+\log\l(\f{1+\sigma}{1-\sigma}\r)^2.
$$

In order to prove that the $\Theta_+$-diameter of $\L\l(C[b,c,\alpha]\r)$ is finite, we just need
to upper bound  $\Theta_+(\L(\fhi),1)$ in the cone $C[b,c,\alpha]$. By the definition of $\Theta_+$ 
this problem reduces to obtain an upper bound to 
\begin{equation}\label{perronemgamaegamacAnosov}
\f{\intgc{\L\fhi\rc}}{\intg{\L\fhi\rho}}=\f{\ds\sum_{j=1}^{p_{\gmc}}\intgcj{\fhi\rcj}}{\ds\sum_{j=1}^{p_{\gm}}\intgj{\fhi\rj}}
\end{equation}
for $\rho \in \Dnorm{\gm}$ and $\rc \in \Dnorm{\gmc}$. 
First, we will bound the expression
$$\f{\intgcj{\fhi\rcj}}{\intgj{\fhi\rj}}
=
\f{\intgcj{\fhi\rcj}}{\intgcj{\fhi}}
\f{\intgcj{\fhi}}{\intgj{\fhi}}
\f{\intgj{\fhi}}{\intgj{\fhi\rj}}$$
for $\gm_j$ and $\gmc_j$ in the same Markov rectangle.  This is the same kind of
calculations beginning with equation \ref{eqrep1} in  Prop. \ref{DiametroFinito}.
From that, we conclude that
$$
\f{\intgcj{\fhi\rcj}}{\intgcj{\fhi}}
\leq
\l(1+b\log\l(\f{1+\lambda}{1-\lambda}\r)\r)\intgcj{\rcj}\\
$$
and
$$
\f{\intgj{\fhi}}{\intgj{\fhi\rj}}
\leq
\f{\l(1+b\log\l(\f{1+\lambda}{1-\lambda}\r)\r)}{\intgj{\rj}}.
$$
Recall that $\rcj = \f{1}{p_{\gmc}}\rc\circ f$ and $\rj = \f{1}{p_{\gm}}\rho\circ f$. Since $\rho$ and $\rc$ are normalized densities, if follows that $\rcj \leq \f{1}{p_{\gmc}}(1+\kappa diam(M)^\alpha)$ and
 $\rj \geq \f{1}{p_{\gm}}(1+\kappa diam(M)^\alpha)^{-1}$. Therefore
$$
\f{\intgcj{\rcj}}{\intgj{\rj}} < \f{p_{\gm}}{p_{\gmc}}(1+\kappa diam(M)^\alpha)^2
$$
On the other hand $\holder{e^{\phi}} < \varepsilon \inf{e^{\phi}}$ and so $\sup{e^{\phi}}<(1+\varepsilon diam (M)^{\alpha})\inf{e^{\phi}}$ . 
So we obtain 
$$
\begin{array}{rcl}
\f{\intgcj{\rcj}}{\intgj{\rj}}  <  \f{p_{\gm}}{p_{\gmc}}(1+\kappa diam(M)^\alpha)^2 
& \leq & p_{max}(1+\kappa diam(M)^\alpha)^2 \\
\end{array}
$$%
Due to  condition (C), the following inequality also holds:
$$
\f{\intgcj{\fhi}}{\intgj{\fhi}} \leq 1+cd\l(\gmcj,\gmj\r)^\alpha \leq 1+c.diam(M)^\alpha
$$
Hence
$$
\f{\intgcj{\fhi\rcj}}{\intgj{\fhi\rj}} < p_{max}\l(1+b\log\l(\f{1+\lambda}{1-\lambda}\r)\r)^2(1+\max\{\kappa,c\} diam(M)^\alpha)^2.
$$
However, given a leaf $\gm$ there could exist more than one leaf $\gm_j$ in the same  Markov rectangle. 
So, we rewrite eq. (\ref{perronemgamaegamacAnosov}) in the following manner:
$$
\f{\intgc{\L\fhi\rc}}{\intg{\L\fhi\rho}}=\f{\ds\sum_{k=1}^{p}\ds\sum_{l=1}^{r_k(\gmc)}\int_{\gmc_{k_l}}\fhi(\rc)_{k_l}\dm_{\gmc_{k_l}}}
{\ds\sum_{k=1}^{p}\ds\sum_{l=1}^{r_k(\gm)}\int_{\gm_{k_l}}\fhi \rho_{k_l}\dm_{\gm_{k_l}}}
$$
such that for $R_k \in \cR$ we have $\gm_{k_l} \subset R_k$ for each $l \in \l\{1,\cdots,r_k(\gm)\r\}$ 
and $\gmc_{k_l} \subset R_k$ for each $l \in \l\{1,\cdots,r_k(\gmc)\r\}$. 
Moreover $\ds\sum_{k=1}^{p}r_k(\gm)=p_{\gm}$ and $\ds\sum_{k=1}^{p}r_k(\gmc)=p_{\gmc}$. Since $\gm_{k_l},\gmc_{k_l} \subset R_k$, by defining 
$$
\tilde{C}:=p_{max}\l(1+b\log\l(\f{1+\lambda}{1-\lambda}\r)\r)^2(1+\max\{\kappa,c\} diam(M)^\alpha)^2,
$$
we have that
$
\f{\ds\int_{\gmc_{k_l}}\fhi(\rc)_{k_l}\dm_{\gmc_{k_l}}}{\ds\int_{\gm_{k_l}}\fhi \rho_{k_l}\dm_{\gm_{k_l}}} \leq \tilde{C}.
$
Now, we are interested in finding an upper bound for
$
\f{\ds\sum_{l=1}^{r_k(\gmc)}\ds\int_{\gmc_{k_l}}\fhi(\rc)_{k_l}\dm_{\gmc_{k_l}}}{\ds\sum_{l=1}^{r_k(\gm)}\ds\int_{\gm_{k_l}}\fhi \rho_{k_l}\dm_{\gm_{k_l}}} 
$.
Suppose $r_k(\gmc) \leq r_k(\gm)$. Since the parcels $\ds\int_{\gm_{k_l}}\fhi \rho_{k_l}\dm_{\gm_{k_l}}$ are positive we can  assume $r_k(\gmc) = r_k(\gm)$ . Therefore, 
$
\f{\ds\sum_{l=1}^{r_k(\gmc)}\ds\int_{\gmc_{k_l}}\fhi(\rc)_{k_l}\dm_{\gmc_{k_l}}}{\ds\sum_{l=1}^{r_k(\gm)}\ds\int_{\gm_{k_l}}\fhi \rho_{k_l}\dm_{\gm_{k_l}}} \leq \tilde{C}
$.
Now, if  $r_k(\gmc) \geq r_k(\gm)$, there exist $q=q(\gm,\gmc),r\in \N$ such that $r_k(\gmc)=q r_k(\gm)+r$, for $0 \leq r<r_k(\gm)$ and $1 \leq q $. So, we write  $l(s)$ for $(s-1)r_k(\gm)+l$ and obtain
$$
\f{\ds\sum_{l=1}^{r_k(\gmc)}\ds\int_{\gmc_{k_l}}\fhi(\rc)_{k_l}\dm_{\gmc_{k_l}}}{\ds\sum_{l=1}^{r_k(\gm)}\ds\int_{\gm_{k_l}}\fhi \rho_{k_l}\dm_{\gm_{k_l}}} =
\ds\sum_{s=1}^{q}\l(\f{\ds\sum_{l=1}^{r_k(\gm)}\ds\int_{\gmc_{k_{l(s)}}}\fhi(\rc)_{k_{l(s)}}\dm_{\gmc_{k_{l(s)}}}}{\ds\sum_{l=1}^{r_k(\gm)}\ds\int_{\gm_{k_l}}\fhi \rho_{k_l}\dm_{\gm_{k_l}}}\r)+
\f{\ds\sum_{l=qr_k(\gm)+1}^{r_k(\gmc)}\ds\int_{\gmc_{k_l}}\fhi(\rc)_{k_l}\dm_{\gmc_{k_l}}}{\ds\sum_{l=1}^{r_k(\gm)}\ds\int_{\gm_{k_l}}\fhi \rho_{k_l}\dm_{\gm_{k_l}}}
$$
Each term in the sum in $s\in \l\{1,\cdots,q\r\}$, satisfies the earlier case therefore the first 
sum can be bounded by $q\tilde{C}$. As $r_k(\gmc)-qr_k(\gm) = r < r_k(\gm)$, the last   parcel also satisfies the first case. In  short, 
$$
\f{\ds\sum_{l=1}^{r_k(\gmc)}\ds\int_{\gmc_{k_l}}\fhi(\rc)_{k_l}\dm_{\gmc_{k_l}}}{\ds\sum_{l=1}^{r_k(\gm)}\ds\int_{\gm_{k_l}}\fhi \rho_{k_l}\dm_{\gm_{k_l}}} \leq (q+1)\tilde{C}.
$$
Note that $q=q(\gm,\gmc) \leq p_{max}$, for any  $\gm$ and $\gmc$. 
Therefore  (\ref{perronemgamaegamacAnosov}) is bounded by  $(p_{max}+1)\tilde{C}$ and  
$$
\Delta:=\sup\l\{\Theta\l(\L\fhi,\L\psi\r); \fhi,\psi \in C(b,c,\alpha)\r\}\leq 2(p_{max}+1)\tilde{C}.
$$
\end{proof}

\subsection{Statistical properties}

Now we construct  a measure in $\Lambda$ which is a good candidate for maximal entropy measure. 
For such probability measure, we prove the  exponential decay of  correlations for H\"older 
continuous  observables and the Central Limit Theorem. 

Up to now,  we have obtained an $\L$-invariant cone $C[b,c,\alpha]$ such that $\L(C[b,c,\alpha])$ has
finite diameter. This guarantees, in some sense the convergence of functions in  
the cone $C[\sigma b,\sigma c,\alpha]$. 

Let $\hat{\eta}$ be any probability measure in the quotient space of $\Lambda$ 
given by $\F^s_{loc}$, and define the measure
\begin{equation}
\eta(\fhi):=\ds\int\l(\intg{\fhi}\r)d\hat{\eta}(\gm)
\label{eqprod}
\end{equation}
For $\fhi \in C[b,c,\alpha]$ let 
$$
\fhi_n=\f{\L^n(\fhi)}{\ds\int{\L^n(1)d\eta}}= \L^n(\fhi),
$$
As $\Theta_+(\fhi_m,\fhi_n) \leq \Theta(\fhi_m,\fhi_n)$ it follows that $\fhi_n$ is a $\Theta_+$-Cauchy 
sequence. 
Moreover, observe also that $\Theta_+(\fhi_n, 1)= \Theta_+(\fhi_n, \L^n(1)) \to 0$.
In particular 
\begin{equation}
\frac{\intg {\fhi_n} }{\int_{\hat \gamma} \fhi_n d\mu_{\hat \gamma}} \to 1, \forall \gamma, \hat \gamma \in  \cF^{s}_{loc}, \text{ uniformly } \label{eqmassa}
\end{equation}
Note that $\intg \fhi$, $\gamma \in \cF^{s}_{loc}$ is bounded from below far from zero, 
because $\fhi \in C([b,c,\alpha])$, say, by a constant $q> 0$. Since $\L$ is positive,
this implies that
$\intg {\fhi_n}  \geq \intg {\L^n(q)} = q$.
We have somewhat more: Suppose that $\mu$ is an $f-$invariant probability satisfying a product measure 
property just as $\eta$ in equation  \ref{eqprod}. Then, for every $n$ there exist 
$\gamma_n$, $\hat \gamma_n$  such that $\int \fhi_n d\gamma_n \geq \int_M \fhi_n d\mu= \int_M \fhi d\mu \geq \int \fhi_n d\hat\gamma_n$. Due to equation \ref{eqmassa},  this means
(even if $\eta$ is not  invariant) 
 the following limit
$$
\limn \ds\int \fhi_n d\eta= \mu(\fhi).
$$
holds.

Let us see that there exists an $f$-invariant probabilty measure  $\mu= \mg \times \hat{\mu}$, which will coincide with the last limit, for any probability $\eta$ satisfying eq.(\ref{eqprod}).
Consider in $M$ the following : $x \sim y$ if, and only if,  $x$ and $y$ belongs in the same leaf 
$\gm \in \F^s_{loc}$.  Let $g:\F^s_{loc} \to \F^s_{loc}$ be the {\em quotient map}  \label{defquot} of $f$ defined by $g(\widetilde{x})=\widetilde{f(x)}$. 
 Take the sigma-algebra generated by the Markov rectangles and its refinements by $f^{-1}$.  It may be that such $\sigma-$algebra
 does not coincide with the Borel $\sigma-$ algebra in $M/\sim$. However, just as in \cite{Cas02}, 
 most of the cylinders in the construction of $\cR \vee f^{-1}(\cR) \dots $ have arbitrarily small $M/\sim$-diameter.
 We give the following construction for the measure $\hat \mu$. Each rectangle has the same measure, 
 and the sum of it is one. Take a rectangle, say, $R_1 \in \cR$. Let $\cR_1:= \{R_{1, 1}, \dots R_{1, n_1}\}$ be the set of 
 connected intersections of $f^{-1}$ and the several elements of $\cR$. Then we set $\hat\mu(R_{1, 1})= \dots = \hat\mu(R_{1, n_1}):= 1/n_1$.
 We continue inductively, taking the pre-images of the elements in $\cR_1$, doing the mass distribution.
 Note that due to the mild mixing property of the Markov Partition, there is $0< c< 1$ such that the cylinders in each phase of the construction have its $\hat\mu-$measure 
 multiplied by a fraction less than $c$ if compared with the cylinders, in the phase before. In particular, 
 if we take a minimal element $S$ in $M/\sim$ with non-zero diameter, its $\hat\mu-$measure can be approached by 
 a cylinder with arbitrarily small $\hat\mu-$measure. This means that $\hat\mu(S)= 0$, and so, $\vee_{n= 0}^{\infty} f^{-n}(\cR)$
 is the Borelian $\sigma-$algebra, modulo null-$\hat\mu$ sets. Note that (by arguments as in \cite{Cas02} ) 
 there are leaves that can be written as enumerable intersections of  cylinders.
 We saturate positively $\hat\mu$ to extend it to a measure in $M$. A simple calculation leads to the fact that 
 $$
 \mu=  \mu_\ga \times \hat \mu,
 $$
 which is invariant by construction.

If one has uniqueness of the maximal entropy measure in the quotient map $g$ defined in the last paragraph, then, 
 one can proceed as in section \ref{secuniqueness} to conclude that the measure that we just
 constructed is the maximal entropy measure for the system in the second setting. 
 In the case when the central-unstable direction is one dimensional, such existence and uniqueness of the 
 maximal entropy measure for the  quotient map is guaranteed by the work of Liverani-Saussol-Vaienti \cite{LSV98}. This implies Corollary \ref{CorLSV}.

Now we proceed with the proof of the exponential decay of correlations for $\mu$.

Again, we have:
\begin{proposition}
If $f$ is invertible then $\L^*(\mu)=\mu$ if and only if $\mu$ is $f$-invariant.
\end{proposition}
Therefore,  $\mu$ is an $f$-invariant probability. 
\begin{equation}\label{dualidadeAnosov}
\int{\l(\fhi \circ f^n\r)\psi}\dm=\int{\fhi\L^n(\psi)}\dm.
\end{equation}
We also have that any  H\"older function can be written as a sum of functions in the cone. That is,

\begin{proposition}
For all $\fhi \in C^\alpha\l(M\r)$ there exists $K(\fhi)>0$ such that $\fhi+K(\fhi) \in C[b,c,\alpha]$.
\end{proposition}
The proof of the proposition above is essentially the same of that of  Lemma \ref{lesoma}.

Now, we are able to prove:
\begin{theorem}[Exponential decay of correlations for systems in the second setting]
The measure $\mu$ exhibits exponential decay of  correlations for H\"older continuous observables.
\end{theorem}

\begin{proof}
The proof is exactly the same of Theorem \ref{teoB}, using the corresponding 
propositions \ref{InvarianciaDoConeAnosov} and \ref{DiametroFinitoAnosov} 
that we have just proved for systems in the second setting.

\end{proof}

The Central Limit Theorem follows again from the exponential decay of correlations and Gordin's Theorem (Th. \ref{TeoremaDeGordin}).

\textbf{Acknowledgements.} The authors are grateful to C. Liverani and P. Varandas for very fruitful conversations 
on thermodynamical formalism and to the anonimous 
referees for the careful reading of the manuscript and suggestions. 
This work was partially supported by CNPq-Brazil and FAPESB.

%
%
\bibliographystyle{alpha}

\end{document}